\UseRawInputEncoding
\documentclass{amsart}

\usepackage{a4wide}
\usepackage{amscd}
\usepackage{amssymb}
\usepackage{cite}
\usepackage{amsmath}
\usepackage{comment}
\usepackage{bookmark}
\usepackage{mathtools}
\usepackage{enumerate}
\usepackage{diagbox}
\usepackage{caption}
\usepackage[usenames, dvipsnames]{color}
\usepackage{color}
\usepackage{hyperref}					
\hypersetup{colorlinks,
	linkcolor=blue,
	citecolor=blue}

  \usepackage{mathrsfs}
\newtheorem{theorem}{Theorem}[section]

\newtheorem{lem}[theorem]{Lemma}
\newtheorem{lemma}[theorem]{Lemma}

\newtheorem{cor}[theorem]{Corollary}
\newtheorem{corollary}[theorem]{Corollary}
\newtheorem{quest}[theorem]{Question}

\newtheorem{proposition}[theorem]{Proposition}

\newtheorem{rem}[theorem]{Remark}
\newtheorem{example}[theorem]{Example}

\newtheorem{defn}[theorem]{Definition}
  \usepackage{makecell}
\newcommand{\norm}[1]{\left\lVert#1\right\rVert}
\newcommand{\hto}{\hookrightarrow }

\usepackage{tikz}
\numberwithin{equation}{section}

\title[Isomorphic classification of $L_{p,q}$-spaces, II]{Isomorphic classification of $L_{p,q}$-spaces, II}

\author[J. Huang]{Jinghao Huang}
\address{School of Mathematics and Statistics, University of New South Wales,
Kensington, 2052, Australia}
\email{\textcolor[rgb]{0.00,0.00,1}{jinghao.huang@unsw.edu.au}}

\author[F. Sukochev]{Fedor Sukochev}
\address{School of Mathematics and Statistics, University of New South Wales,
Kensington, 2052, Australia; North-Ossetian State University, Vladikavkaz, Russia, 362025}
\email{\textcolor[rgb]{0.00,0.00,1}{f.sukochev@unsw.edu.au}}

\subjclass[2010]{46E30. \hfill 
}

\begin{document}

\maketitle
\begin{abstract}
 This is a continuation of the papers \cite{KS} and \cite{SS2018}, in which the isomorphic classification of $L_{p,q}$, for   $1< p<\infty$, $1\le q<\infty$, $p\ne  q $, on resonant   measure spaces, has been obtained.
 The aim of this paper is to give a complete  isomorphic classification of $L_{p,q}$-spaces on general $\sigma$-finite   measure spaces.
 Towards this end, several new subspaces of $L_{p,q}(0,1)$ and $L_{p,q}(0,\infty)$ are identified and studied.
\end{abstract}

\section{Introduction}

This paper is devoted to the isomorphic classification of Lorentz spaces $L_{p,q}$ on  general $\sigma$-finite measure spaces, where $1<  p<\infty$, $1\le q<\infty$, $p\ne q $,
 which is a continuation of previous studies on the isomorphic classification of $L_{p,q}$ on resonant measure spaces contained in \cite{KS,SS2018}
and is motivated by earlier work on subspaces of $L_{p,q}$ due to  Carothers and Dilworth
(see e.g. \cite{Dilworth,CD89,CD,CD85,CF}).
For the isomorphic classification of $L_p$-spaces (i.e., $\ell_p^n$,  $n=1,2,\cdots,$ $\ell_p$, $L_p(0,1)$), we refer to  \cite[Part III]{Wojtaszczyk} and
\cite[Chapter XII]{Banach}.
 For the isomorphic classification of  weak $L_p$-spaces, we refer to
  \cite{Leung,Leung2,Leung3,LS}.

  The Lorentz spaces $L_{p,q}$  were introduced by G.G. Lorentz in \cite{Lorentz50,Lorentz51}, and their  importance is  demonstrated  in several areas of analysis such as harmonic analysis, interpolation theory, etc. (see e.g. \cite{Dilworth,Bennett_S} and references therein).
Recall that if $(\Omega,\Sigma, \mu)$ is a measure space, then for $1<p <\infty$ and $1\le q<\infty$, the Lorentz space $L_{p,q}(\Omega)$ is the collection of all measurable functions $f$ on $\Omega$ such that
$$ \norm{f}_{p,q}:= \left(  \int_{0}^\infty f^* (t) ^q d t^{q/p} \right)^{1/q}<\infty ,$$
where $f^*$ denotes the decreasing rearrangement of $|f|$ (see the next section).
It is well-known that $L_{p,q}(\Omega)$ is separable when $\Omega$ is $\sigma$-finite.
In the special  case when  $(\Omega, \Sigma, \mu)$ coincides with $\mathbb{N}$ equipped with the counting measure, the space $L_{p,q}(\Omega)$ coincides with the familiar symmetric  sequence space $\ell_{p,q}$ \cite{LT1,Dilworth}.
It is well-known that if $1\le q\le p<\infty$, then $\norm{\cdot}_{p,q}$ is a norm, and if $1<p<q<\infty$, then it is a complete quasi-norm which is  equivalent to a complete norm \cite{Dilworth, Bennett_S}.
It is clear that $L_{p,p}(\Omega)$ is the Lebesgue space $L_p(\Omega)$,  and that it is important to note that $L_{p,q}$-spaces arise in the Lions--Peetre $K$-method of interpolation\cite{Dilworth,CD85}.

It is known that $L_{p,q}(0,1)$ and  $L_{p,q}(0,\infty)$ are not isomorphic to each other   for any  $1< p<\infty$, $1\le q<\infty$, $p\ne  q $, and $\ell_{p,q}$ does not embed into $L_{p,q}(0,1)$ as a complemented subspace \cite{CD,Dilworth}.
A stronger result showing that $\ell_{p,q}$ does not embed isomorphically into $L_{p,q}(0,1)$ was proved in \cite{KS} and \cite{SS2018} using the so-called ``subsequence splitting lemma'' introduced and studied in \cite{S96} (see also \cite{DFPS,DDS,DSS,ASS,HSS,SS2004}).
It is interesting to note that there are symmetric function  spaces on $(0,1)$ which contain isomorphic $\ell_{p,q}$-copies, $1<p<2$, $1\le q<\infty$ \cite{Novikova}.
It is also known \cite{KS,SS2018} that
$L_{p,q}(0,\infty)$ does not embed isomorphically into the space $L_{p,q}(0,1)\oplus \ell_{p,q}$.
Hence,
 if $1<p<\infty$, $1\le q<\infty $,  $p\neq q$,  then
$$\ell_{p,q}^n, ~n=1,2,\cdots, ~\ell_{p,q},\; L_{p,q}(0,1),\; L_{p,q}(0,1)\oplus \ell_{p,q}\;\;\mbox{and}\;\; L_{p,q}(0,\infty)$$
is the full  list of pairwise non-isomorphic $L_{p,q}$-spaces over a resonant measure space.

In the present paper, we extend
 results in \cite{SS2018,KS} to the setting of general $\sigma$-finite  measure spaces.
Below, we briefly introduce the structure of the present paper.
Unless stated otherwise,
we will always assume  that $1<p<\infty$, $1\le q<\infty$ and $p\ne q$.

In Section \ref{fr}, we prove some results on the finite representability  of $\ell_{p,q}^n$, which are important auxiliary tools.
In particular, using  a  type/cotype argument and Kadec--Pe{\l}czy\'{n}ski theorem,
 we show that $\ell_{p,q}^n$ does not embed uniformly into $\ell_q$
when $1<p <\min\{2,q\}$ or $p  >\max\{2,q\}$ or $q>2$ (see Corollary~\ref{concl}).

It is well-known  that any $\sigma$-finite measure space is the direct sum of an atomless measure space and an atomic measure space.
Recall that all separable infinite (or finite) atomless  measure spaces are isomorphic to each other \cite[Theorem 9.3.4]{Bogachev},
 and $L_{p,q}(\Omega)$ is isomorphic to $\ell_{p,q}^{n}$, $1\le n\le \infty$,  if $\Omega$ is an atomic resonant  measure space.
 The main   concern of  this paper is the case when  the measures of atoms are not necessarily the same in an atomic measure space.
In Section \ref{atomic}, we
classify infinite-dimensional    $L_{p,q}$-spaces  on different atomic measure spaces and show that any $L_{p,q}$-space on such a measure space  is isomorphic to    a  direct sum of $L_{p,q}$-spaces  of the following  types (see below for definitions and  Section \ref{atomic})
$$\ell_{p,q,1}, ~ \ell_{p,q,0}(I), ~\ell_{p,q,0}(F), ~\mbox{and }\ell_{p,q,\infty}.$$
In Table \ref{t1}, we present a complete isomorphic classification of these spaces.
Note that all Banach spaces $X$ of type $\ell_{p,q,1}$ are   isomorphic to $\ell_{p,q}$.

If the atoms $A_n$ in $\Omega$ satisfy the conditions  that $\mu(A_n)\to 0$ and   $\sum \mu(A_n)=\infty $, then the $L_{p,q}$-space on such a measure space is said to be  of  type $\ell_{p,q,0}(I)$.
 Johnson et.al \cite[p.31]{JMST} introduced the subspace $U_{Y}$ of
 an arbitrary symmetric  function space $Y$ on $[0,\infty)$, which is  spanned by the characteristic functions of
 $A_n$.
  Such spaces are
  natural generalizations  of the space $X_p$ introduced by Rosenthal\cite{Rosenthal}.
  For any fixed symmetric function space $Y$ on $[0,\infty)$,
  it is proved in \cite[Theorem 8.7]{JMST} (see also \cite[Proposition 2.f.7]{LT2}) that
   $U_Y$  does not depend on the particular sequence $\{A_n\}$ used (up to an isomorphism).
   When $Y=L_{p,q}$,
    the space $U_Y$  coincides with an $L_{p,q}$-space of type $\ell_{p,q,0}(I)$ and,  for simplicity,
    we denote   $U_{p,q}:=U_Y=U_{L_{p,q}}$.
By using  techniques different those in \cite{KS,SS2018}, we show that $U_{p,q}$ does not embed into $\ell_{p,q}$ (indeed, it does not embed into $\ell_{p,q}\oplus L_{p,q}(0,1)$, see Theorem \ref{5.10} below),
 which is
a far-reaching generalization of  \cite[Theorem 11]{KS} and \cite[Theorem 10]{SS2018}.
Another interesting case is when the atoms $A_n$ satisfy the condition  $\sum_{n= 1}^\infty \mu(A_n) <\infty$ (or $\mu(A_n)\to_n \infty $).
Any $L_{p,q}$-space on a measure space  of such type is said to be of type $\ell_{p,q,0}(F)$ (resp. $\ell_{p,q,\infty}$).
It is well known that for any disjointly supported sequence of unit vectors   in $L_{p,q}(0,1)$, there exists a subsequence equivalent to the unit vector basis of $\ell_q$\cite[Theorem 5]{Dilworth} (see also \cite{CD}).
We establish  a quantitative version of this result (See Propositions \ref{Flq} and \ref{infiq}), giving a criterion for
 the  normalized characteristic
functions of   atoms (in the setting of
$\ell_{p,q,0}(F)$ and $\ell_{p,q,\infty}$) to be equivalent to the unit vector basis of $\ell_q$.
We also show that  $\left(\oplus_{n=1}^\infty \ell_{p,q}^n\right)_q$ is
isomorphic to a space of the  type $\ell_{p,q,0}(F)$ (or $\ell_{p,q,\infty}$)  for some particular choice of $\{A_n\}$.
In particular, we obtain that  $\left(\oplus_{n=1}^\infty \ell_{p,q}^n\right)_q$ is a subspace  of $L_{p,q}(0,1)$.

In Section \ref{general}, we consider $L_{p,q}$-spaces on arbitrary $\sigma$-finite measure spaces with non-trivial atomless part.
  Up to isomorphism, any such $L_{p,q}$-space is one of the following
$$L_{p,q}(0,1), ~L_{p,q}(0,1)\oplus \ell_{p,q,1 }, ~L_{p,q}(0,1)\oplus \ell_{p,q,0}(I), ~L_{p,q}(0,1)\oplus \ell_{p,q,\infty}, ~L_{p,q}(0,\infty). $$
We give a full characterization of the isomorphic embeddings between $L_{p,q}$-spaces of  all  types listed above.
It is clear that any space of the first four types is a complemented subspace of $L_{p,q}(0,\infty)$.
We show that $L_{p,q}(0,\infty)$ does not embed into any space of the first four types, which extends \cite[Theorem 11]{KS} and \cite[Theorem 10]{SS2018} significantly.
We also  establish an
$L_{p,q}$-space version of a Theorem for Orlicz sequence spaces due to  Lindenstrauss and Tzafriri\cite[Theorem 2.c.14]{LT2} (see \cite[Ch. XII, Theorem 9]{Banach} for the case of $L_p$-spaces), showing that
there exists an isomorphic embedding
$$T:L_{p,q}(0,1)\to  U_{p,q}$$ if and only if $p=q =2$.
In the language of graph theory, the tree below is the Hasse diagram for the partially ordered set consisting of the equivalence classes of $L_{p,q}(\Omega)$ under Banach isomorphism with the order relation.
  For any spaces $X\ne Y$   listed in the tree following, $X$ is isomorphic to a subspace (indeed,   a complemented subspace) of $Y$ if and only  if $X$ can be joined to $Y$ through a descending branch (see Table \ref{t2}).
Note that $(\oplus_{n=1}^\infty \ell_{p,q}^n )_q\hookrightarrow  \ell_q$ for some specific values of $p$ and $q$.

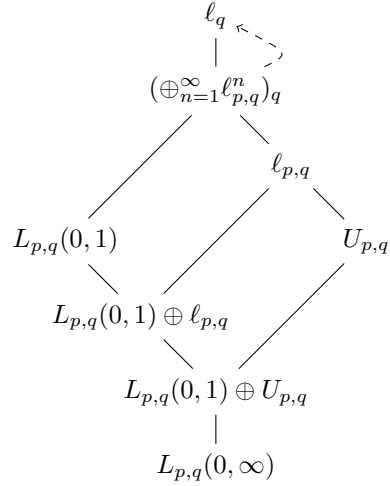
\begin{center}
\begin{tikzpicture}
  \node (max) at (0,3) {$\ell_q$};
  \node (a) at (0,2) {$(\oplus_{n=1}^\infty \ell_{p,q}^n )_q$};
  \node (b) at (-2,0) {$L_{p,q}(0,1)$};
  \node (c) at (1,1) {$\ell_{p,q}$};
  \node (d) at (2,0) {$U_{p,q}$};
  \node (e) at (-1,-1) {$L_{p,q}(0,1)\oplus \ell_{p,q}$};
  \node (f) at (0,-2) {$L_{p,q}(0,1)\oplus U_{p,q}$};
  \node (min) at (0,-3) {$L_{p,q}(0,\infty )$};
  \draw[->,dashed] (a) .. controls (1,2.5) .. (max);
  \draw (max) -- (a) -- (c)
    (a) -- (b)
     (e) -- (c)--(d)
     (b) --(e) --(f) -- (min)
    (d)-- (f)
  ;
\end{tikzpicture}
\end{center}


Our notations and terminology are standard and all unexplained terms may be found in \cite{AK,LT1,LT2,JMST}.

We would like to thank Professor
J. Arazy for helpful discussions concerning results presented in  his paper\cite{Arazy81}, which contains useful techniques in the study of isomorphic embedding of Banach spaces.
We also  thank Professor  W.B. Johnson for his comments on finite representability of $\ell_p$-spaces and   his help in proving Proposition \ref{3.4}.
We thank A. Kuryakov for many joint discussions,
and  T. Scheckter for his careful reading of this paper,
and Professor   E. Semenov
 for useful comments  and sharing the paper~\cite{Bylinkina},
 and we  thank  D. Zanin for helpful discussions.
The second author was   supported by  the Australian Research Council  (FL170100052).
Authors thank the anonymous referee for reading the paper carefully
and providing thoughtful comments, which improved the exposition of the paper.

\section{Preliminaries}\label{prel}

Let $(\Omega, \Sigma,\mu)$ be a measure space with a $\sigma$-finite measure $\mu$,
defined on $\sigma$-algebra $\Sigma$, and let $L(\Omega)$ be the algebra of all classes of equivalent measurable real-valued functions on $(\Omega, \Sigma, \mu)$.
For any function $f\in L(\Omega)$, its distribution $d_f(s)$ is given by
$$d_{f}(s): =\mu (\{f> s\}), ~s>0. $$
Denote by $L_0(\Omega)$ the subalgebra of $L(\Omega)$ consisting of all functions $f$ such that $d_{|f|}(s)<\infty$
for some $s>0$.
For every $f\in L_0(\Omega)$, its non-increasing rearrangement is  defined by
$$ f^*(t):= \inf \{s> 0 :  d_{|f|}(s)\le t\}, ~>0.  $$
In the special case when the measure space is $\mathbb{N}$ of all natural numbers equipped with the counting measure, we denote by $\ell_{p,q}$
the  Lorentz sequence space $L_{p,q}(\mathbb{N})$.
In particular,
 $$\norm{\{a_k\}_{k\ge 1}}_{p,q}=\left(\sum_{k\ge 1} (a^*_k)^q ((k)^{q/p}  - (k-1)^{q/p}) \right)^{1/q}< \infty  ,$$
  where $\{a_k^*\}_{k\ge 0}$ is the decreasing rearrangement of the sequence $\{|a_k|\}_{k=1}^\infty \in \ell_{p,q}$.
  We denote by $\ell_{p,q}^n$ when $\Omega $ consists of $n$'s  atoms of measure $1$.
The mapping $i(\{a_k\}_{k\ge 1}):= \sum_{k=1}^\infty  a_k \chi_{[k-1,k)}$, $\{a_k\}_{k\ge 1}\in \ell_{p,q}$, where $\chi_A$ is the indicator function of a measurable set $A\in \Sigma$, defines an isometric embedding of the space $\ell_{p,q}$
into $L_{p,q}(0,\infty)$.
Note that
for any $\sigma$-finite measure space $(\Omega, \Sigma, \mu)  $,
$L_{p,q}(\Omega)$ embeds isomorphically into $L_{p,q}(0,\infty)$.

Recall that two Banach spaces $X$ and $Y$ are said to be isomorphic (denoted by $X\approx Y$) if there exists an invertible bounded linear  operator from $X$ onto $Y$. Otherwise, we write $X\not \approx   Y$.
We say that $Y$
embeds into
$X$
isomorphically if there is a linear subspace
$Z\subset X$,
such that $Z$
is a Banach space and $Z$
is isomorphic to $Y$.
If   $X$ is isomorphic to a subspace (resp. complemented subspace) of  $Y$,
we write $X\hookrightarrow Y$ (resp. $X\xhookrightarrow{c } Y$).
If $X$ does not embed into $Y$
isomorphically as a subspace  (resp. a complemented subspace), we write $X\not \hookrightarrow Y$ (resp. $X\not \xhookrightarrow{c} Y$).

If two sequences $\{a_k\}$ and $\{b_k\}$
of real positive numbers satisfy the condition
$$
0 < \inf_k \frac{a_k}{b_k} \le \sup_k \frac{a_k}{b_k} <\infty,
$$
we then write $\{a_k\}\sim \{b_k\}$.
We use the same notation to denote the equivalence between basic sequences and this should not cause any confusion.
A sequence $\{x_k\}$
in a Banach space $X$
is called semi-normalized,
if there are scalars
$0<a<b<\infty$
such that $a\le \norm{x_n }\le b$,
$n \ge 1$.
We denote the unit vector basis of $\ell_q$ and
$\ell_{p,q}$
by
$\{e^{\ell_q}_k \}_{k=1}^\infty $
and $\{e_k^{\ell_{p,q}}\}_{k=1}^\infty$, respectively.

We recall below a well-known result (see e.g. \cite[Lemma 2.1]{CD}, see also  \cite{ACL, Dilworth,KamMal,LT1,Tradacete}).
\begin{lem}\label{2.1}
Suppose that $0<p,q<\infty$.
Let $\{f_n\}$ be a sequence of unit vectors in $L_{p,q}(0,\infty)$ such that $f^* \to 0$ pointwise.
Then,
there exists a  subsequence of $\{f_n\}$ which
is equivalent to the unit vector basis of $\ell_q$.
\end{lem}

We end this section with the discussion on  partial averaging operators.
Suppose that $ A_n \subset  \Omega $, $n\ge 1$,  with  $A_n \cap A_m = \varnothing$ as $n \neq m$, $A_n \in \Sigma$, $\mu(A_n)<\infty$. Let $\mathcal{A}$ be a $\sigma$-subalgebra of $\Sigma$, generated by sets $A_n$, $n \in \mathbb{N}$.
The operator $T_{\mathcal{A}}:L_{1}(\Omega) +L_\infty (\Omega) \rightarrow L_{1}(\Omega)+L_\infty (\Omega) $  defined by
$$T_{\mathcal{A}}(f)=\sum\limits_{n=1}^{\infty} \Bigl(\frac{1}{\mu(A_n)}\int\limits_{A_n} \!\!\!fd\mu\Bigr) \chi_{A_n}   ,~ f\in L_1(\Omega)+L_\infty (\Omega) , $$
is called an  averaging operator (see \cite[Ch. 2, \S3]{KPS} or \cite[Section 3.6]{LSZ}).

By \cite[Chapter 4, Lemma 4.5]{Bennett_S} and \cite[Chapter II, Theorem 4.3]{KPS} (see also \cite[Lemma 3.6.2]{LSZ}),   $T_\mathcal{A}$ is continuous on  $L_{p,q}(\Omega )$  when $p\in (1,\infty )$ and $  q\in [1,\infty)$.
That is, any averaging operator is  indeed a continuous projection on $L_{p,q}(\Omega)$.
In particular,
$L_{p,q}(\Omega)$ is isomorphic to a complemented subspace of $L_{p,q}(0,\infty)$.

\section{Some results on the uniform embedding of $\ell_{p,q}^n$}\label{fr}

Recall that $\ell_{p,q}^n$ is  said to embed uniformly in a Banach space $X$ if for every $n\in \mathbb{N}$, there exists  an operator $T_n:\ell_{p,q}^n \to X$ such that $\sup_n \norm{T_n}\norm{T^{-1}}<\infty $.
The connection between type/cotype and the uniform embedding of $\ell_{p}^n$ into a Banach space has been widely studied  (see \cite{MP} and  \cite[Theorem 3.5]{Maligranda}).

It is well-known that $\ell_p\not\hookrightarrow L_{p,q}(0,\infty)$\cite{Dilworth,CD}.
However, $\ell_p$ is finitely representable  in $\ell_{p,q}$, and therefore, in $L_{p,q}(0,\infty)$.
In particular,  Carothers and Flinn\cite{CF} obtained  a quantitative
result in the sense that $\ell_{p}^{n^\alpha}$,  $0<\alpha<1$,  can be embedded into $\ell_{p,q}^n$ with an isomorphism constant $C(p,q,\alpha)$ depending on $p,q,\alpha$ only.

In the propositions below, we study    the 
uniform embedding of $\ell_{p,q}^n$ into  $\ell_q$ when $1< p<\min\{2,q \} $ and $p>\max\{2,q \}$.
The following proposition adopts the technique used in the proof of equality (26) in \cite{Bylinkina} (see also \cite{Maligranda} for similar techniques). We refer to \cite{AK} for the notion of crudely finite representability.

\begin{proposition}\label{3.11}
 If  $1<p <\min\{2,q \} $ (or $p >\max\{2,q \}$), then $\ell_{p,r}^n$ does not embed into $\ell_q$ uniformly for any $r\ge 1$.
\end{proposition}
\begin{proof}
 Recall that there exists a constant $C=C(p,q,\alpha)<\infty$ such that $\ell_p^k$ is $C$-isomorphic to a subspace of $\ell_{p,r}^n$, where $k=O(n^\alpha)$\cite{CF}.
By \cite[Theorem 3.5]{Maligranda}, we obtain that for $1<p <\min\{2,q \} $ (or $p >\max\{2,q \}$), $\ell_p^k $ does not embed into $\ell_q$ uniformly.
Hence, $\ell_{p,r}^n$ does not embed into $\ell_q$ uniformly for any $r\ge 1$.
\end{proof}



\begin{proposition}
\label{3.4}
Let $q>2$.
 Assume that a  Banach space  $E$ with a symmetric basis and $E\not\approx \ell_2$. 
  If $E$ is  crudely  finitely representable in $\ell_q$, then     $  E\approx  \ell_q$.
  In particular, if  $1<p<\infty$, $1\le  r<\infty$ and
  $p\ne q $ (or $r\ne q$), then $\ell_{p,r}^n$   does  not embed into $\ell_q$ uniformly.
\end{proposition}
\begin{proof}
By \cite[Proposition 11.1.13]{AK}, $E$ has an equivalent norm such that $E$ is finitely representable in $\ell_q$.
Under an equivalent norm, the symmetric basis of $E$ is still symmetric.
Applying \cite[Theorem 11.1.8]{AK}, $E$ is isomorphic  to a subspace of $L_q(0,1)$.
By \cite[Corollary 2]{KP}, $E\approx \ell_2$ or $E$ contains $\ell_q$.
This implies that $E\not\approx \ell_1$ \cite[Corollary 2.1.6]{AK}.
Hence, the symmetric basis of $E$ is weakly null.
By Kadec--Pe{\l}czy\'{n}ski theorem \cite[Corollary 5]{KP}, the symmetric basis of $E$ is equivalent to $\ell_2$ or $\ell_q$.
Since
$E\not\approx \ell_2$,  
it follows that $E\approx \ell_q$.

The second assertion follows from the fact that
 $\ell_{p,r}\not\approx \ell_q$ whenever $p\ne q $ or $r\ne q$ (see e.g. \cite[Theorem 4.e.5]{LT1}).
\end{proof}

 It is well-known that
  when $0< q<p<2$, $\ell_q$ contains $\ell_p^n$'s uniformly (see e.g. \cite{JS}, \cite{JS03}, \cite{FLM}, \cite{Kadec} and \cite{K}).
However,   for   other values of $p$ and $q$, the situation becomes different.
Propositions \ref{3.11} and \ref{3.4} indeed  extend this result to the setting of $\ell_{p,q}$ sequence spaces  (see also \cite[Theorem 3.5]{Maligranda}).
\begin{cor}\label{concl}
\begin{enumerate}
  \item If $2< r<\infty$,
then  $\ell_s^n$
embeds uniformly into $\ell_r$ for all $n\ge 1 $ if and only if $s=r$ or $s=2$.
  \item if  $1<s <\min\{2,r \} $ or $s  >\max\{2,r\}$ (or $r>2$ and $s\ne r$),
 then   $\ell_{s, q}^n$
does not embed into $\ell_r$ uniformly for any $q\ge 1$.
In particular, $\ell_{s}^n$
does not embed into $\ell_r$ uniformly.
\end{enumerate}
\end{cor}
\begin{proof}
(1) The necessity follows from Proposition \ref{3.4}.

Now, we prove the sufficiency.  When $s=r$, it is obvious.
If $s=2$, then the assertion follows from Dvoretzky's Theorem\cite[Theorem 11.3.13]{AK}: $\ell_2$ is finitely representable in every infinite-dimensional Banach space.

(2)
The   assertions follow  from
Propositions \ref{3.11} and   \ref{3.4}.
 \end{proof}

\begin{rem}\label{rem:s}
By \cite{Schutt}, $L_{p,1}(0,1)$, $1\le p<2$, is isomorphic to a closed subspace $Y$ of $ L_1(0,1)$.
Since $L_{p,1}(0,1)$ is separable\cite{Bennett_S}, it follows from    \cite[Theorem 11.1.8]{AK} that  $Y$ is finitely representable in $\ell_1$.
Therefore, $\ell_{p,1}^n$'s   uniformly embed in $\ell_1$.
In particular, we obtain a subspace of $\ell_1$ isomorphic to  $\oplus _{\ell_1} (\ell_{p,1}^n)$.
\end{rem}

\section{Isomorphic classification of $L_{p,q}$: the atomic case}\label{atomic}
This  section contains main results of the present paper.
Throughout this section, we always assume that
 $$1< p<\infty , ~ 1\le q<\infty ,~  p\ne q, $$
  and $(\Omega, \Sigma, \mu)$ is an atomic measure space and such that  $\Omega$ consists of atoms $ A_n$,  $1\le n<\infty $,    with $\mu(A_n)<\infty $.
We denote $L_{p,q}(\Omega)$ by $\ell_{p,q}(\{A_n\})$ and the characteristic functions in  $L_{p,q}(\Omega)$ generated by $A_n$ are  denoted by $\chi _{A_n}$.

 \subsection{Decomposition with respect to atoms}
 In this subsection, we introduce notations and  summarize  main results of this section.

There are 6  possible cases for the sequence $\{\mu(A_n)\}_{n=1}^\infty$:
\begin{enumerate}[(i)]
\item $\mu(A_n )\to 0$ as $n \rightarrow \infty$;
  \item $\mu(A_n) \rightarrow \infty$ as $n \rightarrow \infty$;
  \item $0<\lim\limits_{\overline{n \rightarrow \infty}} \mu(A_n) \leq  \overline{\lim\limits_{n \rightarrow \infty}} \mu(A_n) < \infty$;
  \item $0=\lim\limits_{\overline{n \rightarrow \infty}} \mu(A_n) <  \overline{\lim\limits_{n \rightarrow \infty}} \mu(A_n) < \infty$;
  \item $0<\lim\limits_{\overline{n \rightarrow \infty}} \mu(A_n) <  \overline{\lim\limits_{n \rightarrow \infty}} \mu(A_n) = \infty$;
  \item  $0=\lim\limits_{\overline{n \rightarrow \infty}} \mu(A_n) <  \overline{\lim\limits_{n \rightarrow \infty}} \mu(A_n) = \infty$.
\end{enumerate}

\begin{defn}
The space
$\ell_{p,q}(\{A_n\})$ is said to be of
\begin{enumerate}
  \item type $\ell_{p,q,0} $ if it satisfies (i). If, in addition,  $\sum \mu(A_n)<\infty $ (resp.  $\sum \mu(A_n)=\infty $), then it is called of type  $\ell_{p,q,0}(F) $ (resp. $\ell_{p,q,0}(I)$);
   \item type $\ell_{p,q,\infty}$ if it satisfies (ii);
 \item type $\ell_{p,q,1}$ if it satisfies (iii).
\end{enumerate}

\end{defn}

\begin{rem}\label{remark}
Let  $(A_n)_{n=1}^{\infty}$ and  $(B_n)_{n=1}^{\infty}$
 be atoms such that
 $$0<c \leq \frac{\mu(A_n)}{\mu(B_n)} \leq C < \infty$$
for some positive  constants $c$ and  $C$.
Then,
$$\ell_{p,q}(\{A_n\}) \sim \ell_{p,q}(\{B_n\}).$$
Indeed, one only needs  to observe that, for any sequence $(a_n)_{n\ge 1} \in \ell_\infty $ with finite non-zero elements, we have
 \begin{align*}
  \norm{  \sum_{n\ge 1} a_n \chi_{[ \left(\sum_{k\le n}c \mu(B_n)\right)  -c\mu(B_n), \sum_{k\le n} c\mu(B_k))}}_{p,q}
 \le  & \norm{  \sum _{n\ge 1}  a_n \chi_{[\left(\sum_{k\le n} \mu(A_n)\right)  -\mu(A_n),\sum_{k\le n} \mu(A_k))}}_{p,q}\\
  \le &  \norm{  \sum_{n\ge 1} a_n \chi_{[\left(\sum_{k\le n} C\mu(B_n)\right )  -C \mu(B_n),\sum_{k\le n} C\mu(B_k))}}_{p,q} .
 \end{align*}
 This  implies that $\{\chi_{A_n}\}$ and $\{\chi_{B_n}\}$
are equivalent unconditional bases in $\ell_{p,q}(\{A_n\})$ and in $\ell_{p,q}(\{B_n\})$, respectively.
In particular, for any space $\ell_{p,q}(\{A_n\})$ of type $\ell_{p,q,1}$,
 the unconditional basis $\{\chi_{A_n}\}$ of  $\ell_{p,q}(\{A_n\})$ is equivalent to the unit vector basis of   $\ell_{p,q}$.
\end{rem}

We collect  main results of this section below.
The proof will be given later in the following subsections.
For the sake of convenience,
we denote by $A \hookrightarrow B$ (resp. $A\not\hookrightarrow B$, $A \xhookrightarrow{c} B $) if $X  \hookrightarrow Y $ (resp.  $X\not \hookrightarrow Y $, $X \xhookrightarrow{c} Y$) for any   $L_{p,q}$-space $\ell_{p,q}(\{A_n\})$ of type $A$ and any  $L_{p,q}$-space $\ell_{p,q}(\{B_n\})$   of type  $B$.

\begin{theorem}\label{Main}
If $1< p<\infty , ~ 1\le q<\infty ,~  p\ne q $, then
\begin{enumerate}
          \item  $\ell_{p,q,0}(I) \not\hookrightarrow \ell_{p,q} $;
                \item $\ell_{p,q,\infty} \xhookrightarrow{c} \ell_{p,q}\xhookrightarrow{c} \ell_{p,q,0}(I)$  and        $\ell_{p,q,0}(F) \xhookrightarrow{c} \ell_{p,q,0}(I )$;
                    \item $ \ell_{p,q} \not \hookrightarrow \ell_{p,q,\infty} $;
\item    $\ell_{p,q}\not \hookrightarrow \ell_{p,q,0}(F)$;
  \item all spaces of type $\ell_{p,q,0}(I )$ (resp. $\ell_{p,q,1}$) are isomorphic to $U_{p,q}$ (resp. $\ell_{p,q}$);
      \item $\ell_q$ is isomorphic to a complemented subspace of all types: $\ell_{p,q,1}$, $\ell_{p,q,0}(F)$, $\ell_{p,q,\infty}$, $
\ell_{p,q,0}(I)$.
\end{enumerate}

\end{theorem}
\begin{proof}
The assertion (i) follows from Proposition \ref{4.8} below.

The assertions  $ \ell_{p,q}\xhookrightarrow{c} \ell_{p,q,0}(I)$ and $   \ell_{p,q,0}(F) \xhookrightarrow{c} \ell_{p,q,0}(I )$  in (ii)
follow from Proposition \ref{0c}.

The assertion
 $\ell_{p,q,\infty} \xhookrightarrow{c} \ell_{p,q}$ in (ii) follows from Proposition \ref{lpqabsorb}.

The assertions
(iii), (iv) and  (v) follow  from  Proposition \ref{noti},  Proposition \ref{lpqnot}
and Proposition \ref{lpqinftyequ} (and Remark \ref{remark}), respectively.

The assertion  (vi)  is well-known (see e.g. \cite[Corollary 2.4]{CD}).
\end{proof}

 The following corollary shows that any $L_{p,q}$-spaces of type (iv) (or (v) or (vi)) can be decomposed as a  direct sum of $L_{p,q}$-spaces of type $\ell_{p,q,0}$, $\ell_{p,q,1}$ and $\ell_{p,q,\infty}$.
\begin{cor}\label{cor} If $\ell_{p,q}(\{A_n\})$ is of   type (iv) (or (v) or (vi)), then there exists $X_1$ of type $\ell_{p,q,0}$, $X_2$ of type $\ell_{p,q,1}$ and $X_3$ of type $\ell_{p,q,\infty}$ such that
$$\ell_{p,q}(\{A_n\}) \approx  X_1\oplus X_2\oplus X_3. $$
\end{cor}
\begin{proof}
We only prove the  type (vi) case.
All other cases are similar, and
actually,
simpler.

If for any   positive number $a$ such that
$\sum_{ \mu(A_n)< a } \mu(A_n )=\infty  $.
Then, $\ell_{p,q}(\{A_n\})$ is isometric to  $U_{p,q}$\cite[Lemma 8.7]{JMST} (see also Proposition \ref{0c} below).

Now,
assume that there exists a positive number $a$ such that
$\sum_{ \mu(A_n)< a } \mu(A_n )<\infty $.
Then, letting $\{B_n \} = \{A_n:   \mu(A_n)< a\} $, we obtain that $\ell_{p,q}(\{A_n\})\approx \ell_{p,q}(\{B_n\})\oplus \ell_{p,q}(\{A_n\}\setminus \{B_n \})$.
Note that $\ell_{p,q}(\{B_n\})$ is of type $\ell_{p,q,0}(F)$ and $\ell_{p,q}(\{A_n\}\setminus \{B_n \})$ is of type (v).

Now, we may   assume that $\{A_n\}$ satisfies the condition for   type (v).
Without loss of generality, we may assume that
$\lim\limits_{\overline{n \rightarrow \infty}} \mu(A_n) =1$.
Firstly, we consider   the case when  for any positive number $a>0$,  ${\rm Card}(n:\mu(A_n)< a)<\infty$.
In this case,   $\ell_{p,q}(\{A_n\})$
is indeed of type $\ell_{p,q,\infty}$.

Now, assume that there exists $a>0$ such that   ${\rm Card}(n:\mu(A_n)< a)=\infty$.
Denote by $\{k_i\}_{i=1}^N:=\left\{ k_i \in \mathbb{N }:   {\rm Card}(n: k_i \le \mu(A_n)< k_i +1)=\infty\right\}$, where $N \in \mathbb{N}\cup \{\infty\}$.

For the sake of simplicity,
we denote by $\{B_n\}$ the union of all sets $\{  A_n :  k_i\le \mu(A_n)<k_i+1, ~i\ge 1 <N \}$, $i\ge 1$.
By Remark \ref{remark}, we may assume that
$\{B_n\}$ consists of infinitely many atoms of measure $k_i$, $1\le i< N$.
It is clear that   $\ell_{p,q}\xhookrightarrow{c} \ell_{p,q}(\{B_n \})$.
On the other hand,  taking some suitable averaging operators, we obtain   that
$$\ell_{p,q}(\{B_n   \}) \xhookrightarrow{c}  \ell_{p,q}\xhookrightarrow{c} \ell_{p,q}(\{B_n   \} )   .$$
Moreover,
$\ell_{p,q}(  \{B_n\}   )\approx \ell_{p,q}( \{B_n\})\oplus \ell_{p,q}( \{B_n\}  )$ and $\ell_{p,q}\approx \ell_{p,q}\oplus \ell_{p,q}$.
By Pe{\l}czy\'{n}ski decomposition method\cite[Theorem 2.2.3]{AK}, we obtain that
$$\ell_{p,q}(\{B_n   \}) \approx \ell_{p,q}   .$$
Noting that $\ell_{p,q}(\{A_n \}\setminus \{B_n   \})$ is of type $\ell_{p,q,\infty}$, we   complete the proof.
\end{proof}

The following corollary follows immediately  from   Corollary \ref{cor} and Propositions \ref{0c} and \ref{lpqabsorb} below.
Note that there are spaces of type $\ell_{p,q,0}(F)$ (or $\ell_{p,q,\infty}$) isomorphic to $\ell_q$ or $(\oplus_{n\ge 1} \ell_{p,q}^n)_q$. Therefore, there exist  spaces of type $\ell_{p,q,0}(F)$  isomorphic to   spaces of $\ell_{p,q,\infty}$.
\begin{corollary}\label{cortype}Any $\ell_{p,q}(\{A_n\})$ is of one of the following types (up to an isomorphism):
$$\ell_{p,q,1},~ \ell_{p,q,\infty } ,~ \ell_{p,q,0}(F),~ \ell_{p,q,0}(I),  ~ \ell_{p,q,0}(F)\oplus \ell_{p,q,1},~ \ell_{p,q,0}(F)\oplus \ell_{p,q,\infty } . $$
\end{corollary}

Recall that Corollary \ref{cor} shows that any
$\ell_{p,q}(\{A_n\})$ can be decomposed into the direct sum of spaces of type  $\ell_{p,q,0}(F)$ (or $\ell_{p,q,0}(I)$), of type $\ell_{p,q,1}$ and of type $\ell_{p,q,\infty}$.
In the following table, we collect all possible isomorphic types of $L_{p,q}$-spaces of  these fundamental types (i.e., $\ell_q$,  $\ell_{p,q,0}(F)$, $\ell_{p,q,0}(I)$,  $\ell_{p,q,1}$ and  $\ell_{p,q,\infty}$).
Below, `sometimes' means that there exist Banach  spaces $X_1$ and $X_2$  of type $A$ such that $X_1\approx Y$ and $X_2\not \approx Y$ for some space $Y$ of type $B$.

We supply below all the references to Table \ref{t1}.

The fact that $\ell_q \not\approx \ell_{p,q}$ is trivial and folklore (see e.g. Corollary \ref{4.28} and Proposition \ref{lpqabsorb}, see also \cite[Theorem  4.e.5]{LT1});

The assertions $\ell_{p,q,0}(I)\not\approx \ell_q$ and $\ell_{p,q,0}(I)\not \approx \ell_{p,q}$ follow   from Proposition \ref{4.8} and the fact that $\ell_q \hto  \ell_{p,q}$\cite{Dilworth};

For the relations between $\ell_q$ and $\ell_{p,q}(F)$ (resp. $\ell_{p,q,\infty}$), see Theorem \ref{4.19} (resp. Proposition \ref{infiq} and Corollary \ref{4.28});

The assertions  $\ell_{p,q}\not\hto \ell_{p,q,0}(F)$ and $\ell_{p,q,0}(I)\not\hto \ell_{p,q,0}(F)$ (resp. $\ell_{p,q,\infty}$) follow    from Proposition \ref{lpqnot} (resp. Proposition \ref{noti});

The relation between $\ell_{p,q,0}(F)$ and $\ell_{p,q,\infty}$ follows from Theorems \ref{4.19}, \ref{newsubspace},   Proposition \ref{infiq} and Theorem \ref{4.30}.
$\\$
\begin{center}
\begin{tabular}{|c| c|c|c|c|c|   } \hline
 \diagbox[width=6.2em]{$X$}{$Y$}
 & $\ell_q$& $\ell_{p,q}$&  $\ell_{p,q,0}(I )$ & $\ell_{p,q,0}(F)$ & $\ell_{p,q,\infty}$
   \\
  \hline
  $\ell_q$  & $-$ &  Never &  Never & Sometimes &Sometimes\\\hline
 $\ell_{p,q}$ &  Never  & $-$ &Never &  Never & Never
   \\\hline
$\ell_{p,q,0}(I )$ &  Never & Never &  $-$ & Never & Never
     \\\hline
  $\ell_{p,q,0}(F)$ & Sometimes & Never &  Never & $-$ & Sometimes
 \\\hline
 $\ell_{p,q,\infty}$  & Sometimes   & Never & Never & Sometimes  &$-$
  \\\hline
\end{tabular}
\captionof{table}{ When is $X$ isomorphic to  $Y$?}\label{t1}
\end{center}

\subsection{The type $\ell_{p,q,0}(I)$ case}

The following proposition is an immediate consequence of \cite[Lemma 8.7]{JMST} (see also \cite[Proposition 2.f.7]{LT2}).
It shows that all spaces of type $\ell_{p,q,0}(I)$ are isomorphic to each other.
\begin{proposition}\label{lpqinftyequ}For any $1<p<\infty$, $1\le   q<\infty$ and $p\ne q$.
Assume that   $X$ and $Y$ are  of type $\ell_{p,q,0}(I)$.
Then, we have
$$X\approx Y.$$%
\end{proposition}

We adopt the notation in \cite{JMST} and denote by $U_{p,q}$ a space of type $\ell_{p,q,0}(I)$.
The following proposition  shows that $U_{p,q}$
is the `largest' space among the all types of atomic $L_{p,q}$-spaces.
 \begin{proposition}\label{0c}For any $1<p<\infty$, $1\le   q<\infty$ and $p\ne q$.
 Let  a Banach space  $X$ be  of type $\ell_{p,q,1}$, or $\ell_{p,q,0}(F)$, or $\ell_{p,q,0}(I)$ or $\ell_{p,q,\infty}$.
We have
 $$U_{p,q}\oplus X\approx U_{p,q}. $$
 In particular,
  $U_{p,q}$ has a complemented subspace isometric  to $X$.
  \end{proposition}
 \begin{proof}
 Let $U_{p,q}$ be an arbitrary space $\ell_{p,q}(\{A_n\})$ of type $\ell_{p,q,0}(I)$.
 It is obvious that for  $\ell_{p,q}(\{B_n\})$ of type $\ell_{p,q,0}(F)$ (or $\ell_{p,q,0}(I)$ or $\ell_{p,q,\infty}$),
    $\{C_n\}:=\{A_n\}\cup \{B_n \}$ satisfies that
    $$\sum_{ \mu(C_n)\le \varepsilon  }\mu(C_n)=\infty$$
 for any $\varepsilon>0$.
 Hence, by \cite[Lemma 8.7]{JMST},  $ U_{p,q}\oplus  \ell_{p,q}(\{B_n\})\approx  \ell_{p,q}(\{A_n\})\oplus \ell_{p,q}(\{B_n\})\approx \ell_{p,q}(\{C_n\})\approx U_{p,q} $ and therefore, $\ell_{p,q}(\{B_n\})$ is a complemented subspace of $U_{p,q}$.
 \end{proof}

The following proposition distinguishes  $\ell_{p,q}$ from $U_{p,q} $ and  strengthens the well-known result that $L_{p,q}(0,\infty)\not\hto \ell_{p,q}$ (see e.g. \cite{KS,SS2018}).

\begin{proposition}\label{4.8}
Let $1< p<\infty$, $1\le q <\infty $ and $p\ne q$.
Then
$$U_{p,q} \not\hookrightarrow \ell_{p,q}.$$
\end{proposition}
\begin{proof}
By Proposition \ref{lpqinftyequ},
without loss of generality, we may assume that $\{A_n\}=\cup_n B_n$,
where
for each $n\ge 1$,     $B_n=\{A_k^n\}_{k=1}^\infty$ is    a set of infinitely  many atoms of measure $\frac1n$.

Assume that
there exists an isomorphism
$$T:X \hookrightarrow \ell_{p,q} .$$
Let  $ b_n:= \frac{1}{\norm{\chi_{(0,\frac1n) } }}_{p,q} $.
In particular, $b_n\to \infty$ as $n\to \infty$.

Fix a number $\varepsilon_1>0$.
When $n=1$, we define
an element $x_1^1\in \ell_{p,q}$ with finite-rank   support such that
$$\norm {x_1^1 -T(b_1 \chi_{A_1^1})}_{p,q}\le \varepsilon_1.$$
Let $n=2$.
Since $\left\{ b_n \chi_{A^n_k}  \right\}_{k=1}^\infty $ is weakly null, it follows that there exists a subset of   $\{\chi_{A^n_{k_j^{(n)}} }\}_{1\le j\le n }  $ of $\{\chi_{A^n_k}\} _k $ and   mutually disjoint elements   $x_{j}^n\in \ell_{p,q} $, which are also disjoint with $x_1^1$,
such that
$$\norm{ T(b_n \chi_{A_{k_j^{(n)} }^n}) - x_j^n    }_{p,q }\le  \frac{1}{2^{j+1}}\varepsilon_n
 .$$
Arguing inductively, for any given sequence $\{\varepsilon_n\}$, we construct a sequence
$$\left\{
T(b_1 \chi_{A_{1}^1}) , T(b_1 \chi_{A_{k_1^{(2)}}^2}), T(b_1 \chi_{A_{k_2^{(2)} }^2}),   T(b_1 \chi_{A_{k_1^{(3)}}^3}),
\cdots, T(b_1 \chi_{A_{k_1^{(n)}}^n}) ,\cdots T(b_1 \chi_{A_{k_n^{(n)}}^n}) , \cdots
\right\}$$
such that
\begin{align}\label{oneb}
\norm{ T(b_n \chi_{A_{k_{j}^{(n)} }^n}) - x_j^n   }_{p,q }\le  \frac{1}{2^{j+1}}\varepsilon_n
 ,
 \end{align}
where $x_j^n\in \ell_{p,q}$, $1\le j\le n$, $n\ge 1$, are mutually disjoint.
In particular,  for each $n$, we have
\begin{align}\label{sumoneb}\norm{\sum_{j=1}^n T(b_n \chi_{A^n_{k_j}})- \sum_{j=1}^n  x_{j}^n }_{p,q }\stackrel{\eqref{oneb}}{\le}  \varepsilon_n .
\end{align}


Note that
\begin{align}\label{estofchi}
\norm{\sum_{j=1}^n b _n \chi_{A_{k_j^{(n)}}^n}  } _{p,q } =b_n .
\end{align}
We obtain that
$$ \norm{T^{-1}}^{-1 } b_n  \le  \norm{T\left(\sum_{j=1}^n b _n \chi_{A_{k_j^{(n)} }^n} \right ) } _{p,q } \le \norm{T}b_n $$
and, by \eqref{sumoneb}, we have
$$  \norm{T^{-1}}^{-1 }  b_n  -  \varepsilon_n\le  \norm{\sum_{j =1}^n x_j^n  } _{p,q }
\le  \norm{T}b_n   +\varepsilon_n. $$
That is, $  \norm{T^{-1}} ^{-1 }   -  \varepsilon_n/b_n\le  \norm{\frac{\sum_{j =1}^n x_j^n}{b_n}  } _{p,q }
\le  \norm{T}   +\varepsilon_n/b_n . $
By \eqref{sumoneb},
one can choose $\{\varepsilon_n\}_n$ such that (see e.g. \cite[Theorem 1.3.9]{AK})
$$\left\{\sum_{j=1}^n T( \chi_{A^n_{k_j}}) \right \}_n \sim \left\{\frac{\sum_{j =1}^n     x_{j}^n}{b_n}\right \}_n   .$$

Note that $\{x_j^n\}_{n,j}$ are mutually disjoint and uniformly bounded.
Since $b_n\to _n \infty$, we  obtain that
$$\frac{\sum_{k =1}^n x_k^n}{b_n}\to 0\mbox{ in the uniform norm}$$
as $n\to \infty$.
By Lemma \ref{2.1}, there exists a subsequence $\{b_{n_j}\}_j$ of $\{b_n\}_n $ such that
$\left\{\frac{\sum_{k=1}^{n_j}  x_k^{n_j}}{b_{n_j}}\right\}_j  \sim \{e^{\ell_q}_{j}\}_j$.
%
We obtain that
$$j^{1/p}
\sim \norm{  T(\chi_{ (0,j)}) }_{p,q } \sim   \norm{ \frac{ \sum_{k=1}^{ 1}  x_{k}^1}{b_1} \oplus \frac{\sum_{k=1}^{n_2 }  x_{k}^{n_2}}{b_{n_2}} \cdots \oplus \frac{\sum_{k=1}^{n_j} x_k^{n_j}  }{b_{n_j} } }_{p,q } \sim
j ^{1/q } ,$$
which is impossible whenever  $p\ne q$.
\end{proof}

\subsection{The $\ell_{p,q,0}(F)$ case}
In this section,
without loss of generality, we always assume that  $\sum_{n=1}^\infty  \mu(A_n)=1$ and $\mu(A_n )\ge \mu(A_{n+1})$ for every $n$.
Let $b_n :=    \sum_{k=n}^\infty \mu(A_k )   $.
In particular, $\ell_{p,q}(\{A_n\})$ can be identified as the subspace of $L_{p,q}(0,1)$ generated by intervals $[b_{n+1},b_n)$.

The following is a direct consequence of results in \cite{KS,SS2018}.
\begin{proposition}\label{lpqnot}Let $1< q<\infty$, $1\le q<\infty$,  $p\ne q$.
Let $X$ be a space of type $\ell_{p,q,0}(F)$.
Then, $$\ell_{p,q} \not \hookrightarrow X .$$
In particular, $U_{p,q}\not \hookrightarrow X $.
\end{proposition}
\begin{proof}
Assume by contradiction that $  \ell_{p,q}\hookrightarrow X $.
Note that  $X\hookrightarrow L_{p,q}(0,1)$.
Hence, $\ell_{p,q}\hookrightarrow L_{p,q}(0,1)$, which is a contradiction to the fact that $\ell_{p,q}\not \hookrightarrow L_{p,q}(0,1)$ (see \cite[Corollary 8 and Theorem  10]{KS} and \cite[Theorem 9]{SS2018}).
\end{proof}

It is well-known that for any  sequence of  disjointly supported normalized elements   in $L_{p,q}(0,1 )$, there exists a subsequence equivalent to the unit vector basis of $\ell_q$\cite[Theorem 5]{Dilworth} (see also \cite{CD}).
The proposition below provides a quantitative version of this result, providing a criterion  for the sequence of  disjointly supported normalized characteristic   functions  in $L_{p,q}(0,1)$ to be equivalent to the unit vector basis of $\ell_q$.
The main tools used in the proof below are the  Hardy--Littlewood inequality \cite[Chapter II, Theorem 2.2]{Bennett_S} and its generalization \cite[Theorem 8.2]{Luxemburg} (see also \cite[Proposition 2.3]{DDP}).
Note that the following result holds true for $0<p\ne q<\infty$.
\begin{proposition}\label{Flq}Let $1<p<\infty$, $1\le q<\infty$, $p\ne q$.
Let $0<\cdots <b_{n+1} <b_n< \cdots <b_1=1$ and $\mu(A_n):=b_{n}-b_{n+1} \ge \mu(A_{n+1}):= b_{n+1}-b_{n+2}$ for all $n$.
Let $B_n = \norm{\chi_{(b_{n+1},b_n]}}_{p,q} = (b_n-b_{n+1})^{1/p}$. Then,
$\{B_n^{-1}\chi_{(b_{n+1},b_n]}\}$ in $L_{p,q}(0,1)$ is equivalent to $ \{e^{\ell_q}_n\}$ if   $\sup_n \frac{b_{n+1}}{b_n}<1$.
In particular,
$\ell_{p,q}(\{A_n\})\approx \ell_q$ if     $\sup _k  \frac{ \mu(A_{k+1 } )}{  \mu(A_k) }<1$.

On the other hand,
if $\lim_{n\to \infty} \frac{b_{n+1}}{b_n}=1 $,
then
$\{B_n^{-1}\chi_{(b_{n+1},b_n]}\}_{n=1}^\infty $ in $L_{p,q}(0,1)$ is not equivalent to $ \{e^{\ell_q}_n\}$.
\end{proposition}
\begin{proof}
Let $x = \sum_{n\ge 1} a_n B_n^{-1}\chi_{(b_{n+1},b_n]}$ for an arbitrary  sequence $(a_n)$ with finitely many non-zero elements.

Assume first that $p\ge q\ge 1$.
By the Hardy--Littlewood inequality \cite[Chapter II, Theorem 2.2]{Bennett_S}, we obtain that
\begin{align}\label{c1}
\int_0^1x^*(s)^qds^{q/p} \geq \int_0^1 |x(s)|^qds^{q/p}=
\sum_{n\ge 1} |a_n|^q  B_n^{-q} (b_{n }^{q/p}- b_{n+1 }^{q/p}) .
\end{align}
We have
\begin{align*}
\norm{x}_{p,q}^q =\norm{|x|^q}_{p/q,1}   \stackrel{\eqref{c1}}{\ge}\sum_{n\ge 1}  |a_n|^q   B^{-q}_n  (b_{n }^{q/p}- b_{n +1 }^{q /p}) =
\sum_{n\ge 1}  |a_n|^q    \frac{b_{n }^{q/p}- b_{n +1 }^{q /p}}{(b_n-b_{n+1})^{q/p}}.
\end{align*}
Since $\sup_n \frac{b_n}{b_{n+1}}<1$, it follows that $\inf_n \frac{b_{n }^{q/p}- b_{n+1 }^{q /p}}{(b_n-b_{n+1})^{q/p}}  >0$.
We obtain that
$$\norm{x}_{p,q}   \ge \left(\inf_m \frac{b_{m }^{q/p}- b_{m+1 }^{q /p}}{(b_m-b_{m+1})^{q/p}}\right)^{1/q } \cdot \norm{(a_n)}_{q}.$$
On the other hand, by the $p$-convexity of $L_{p,q}$-spaces (see e.g. \cite[Eq.(17)]{Dilworth} for the case when $1\le q\le p<\infty$ and the case when  $q<  1$  follows from   \cite[Proposition 3.1.(6)]{DDS2014}), we have
$$\norm{x}_{p,q}   \le \left( \sum_{n\ge 1}  \norm{   a_n B_n^{-1}\chi_{(b_{n+1},b_n]}  }^q_{p,q} \right)^{1/q}=\norm{(a_n)}_q  .$$
We conclude that $\{B_n^{-1}\chi_{(b_{n+1},b_n]}\}_n \sim \{e^{\ell_q}_n\}_n $.

Assume that $\lim_n \frac{b_{n+1}}{b_n}=1$.
Let $x_n:=\sum_{k\ge n } B_k B_k^{-1} \chi_{(b_{k+1},b_k]}=\sum_{k\ge n }  \chi_{(b_{k+1},b_k]}$.
Note that
\begin{align*}
\norm{x_n^q}_{p/q,1} = \int_0^1 |x_n (s)|^q ds^{q/p}= \sum_{k\ge n } |B_k |^q    \frac{b_{k }^{q/p}- b_{k +1 }^{q /p}}{(b_k-b_{k+1})^{q/p}} \le  \sup _{k\ge n } \frac{b_{k }^{q/p}- b_{k +1 }^{q /p}}{(b_k-b_{k+1})^{q/p}}  \sum_{k\ge n } B_k   ^{q}.
\end{align*}
Since $\lim_{n\to\infty} \frac{b_{n+1}}{b_n}=1$, it follows that
$$\left(
\frac{\norm{x_n }_{p ,q} }{ \norm{(B_k )_{k\ge n }}_q   }
\right)^{ q }= \frac{\norm{x_n^q}_{p/q,1} }{ \norm{(B_k )_{k\ge n }}_q^q  } \le  \sup _{k\ge n } \frac{b_{k }^{q/p}- b_{k +1 }^{q /p}}{(b_k-b_{k+1})^{q/p}} \to 0,~\mbox{as $n\to\infty$}. $$
This implies that
$\{B_n^{-1}\chi_{(b_{n+1},b_n]}\}$ in $L_{p,q}(0,1)$ is not equivalent to $ \{e^{\ell_q}_n\}$ in $\ell_q $.

Now, we assume that $q>p >1 $.
By \cite[Theorem 8.2]{Luxemburg}, we obtain that
\begin{align*}
\int_0^1x^*(s)^q ds^{q/p} &\le  \int_0^1 |x(s)|^q ds^{q/p}=
\sum_{n\ge 1} |a_n|^q  B_n^{-q} (b_{n }^{p/q}- b_{n+1 }^{q/p })\\
&= \sum_{n\ge 1} |a_n|^q    \frac{b_{n }^{q/p}- b_{n +1 }^{q /p}}{(b_n-b_{n+1})^{q/p}}= \norm{(a_n)}_q^q  \sup_{n}   \frac{b_{n }^{q/p}- b_{n +1 }^{q /p}}{(b_n-b_{n+1})^{q/p}}.
\end{align*}
Since $\sup_n \frac{b_n}{b_{n+1}}<1$ and $\frac{q}{p} >1$, it follows that $\inf_n \frac{b_{n }^{q/p}- b_{n+1 }^{q /p}}{(b_n-b_{n+1})^{q/p}}  <\infty  $.
This shows that $\norm{x}_{p,q} \gtrsim \norm{(a_n)}_q$.
On the other hand, by the $q$-concavity (see e.g. \cite[Theorem 3]{Dilworth} for the case when $1< p<q  <\infty$ and the case when $p\le 1$   follows from   \cite[Proposition 3.1.(6)]{DDS2014}), we have
$$
\norm{x}_{p,q}
\ge
\left( \sum_{i=1}^\infty \norm{a_n B_n^{-1}\chi_{(b_{n+1},b_n]}}_{p,q}^q \right)^{1/q}
=
\left( \sum_{i=1}^\infty   a_n ^q \right)^{1/q}=
\norm{(a_n)}_q .
 $$

Assume that $\lim_n \frac{b_{n+1}}{b_n}=1$.
Let $x_n:=\sum_{k\ge n } B_k B_k^{-1} \chi_{(b_{k+1},b_k]}=\sum_{k\ge n }  \chi_{(b_{k+1},b_k]}$.
Note that
\begin{align*}
\norm{x_n^q}_{p/q,1} = \int_0^1 |x_n (s)|^q ds^{q/p}
= \sum_{k\ge n } |B_k |^q    \frac{b_{k }^{q/p}- b_{k +1 }^{q /p}}{(b_k-b_{k+1})^{q/p}} \ge  \inf _{k\ge n } \frac{b_{k }^{q/p}- b_{k +1 }^{q /p}}{(b_k-b_{k+1})^{q/p}}  \sum_{k\ge n } B_k   ^{q}.
\end{align*}
Since $\lim_{n\to\infty} \frac{b_{n+1}}{b_n}=1$, it follows that
$$\left(
\frac{\norm{x_n }_{p ,q} }{ \norm{(B_k )_{k\ge n }}_q   }
\right)^{ q }= \frac{\norm{x_n^q}_{p/q,1} }{ \norm{(B_k )_{k\ge n }}_q^q  }  \le  \inf  _{k\ge n } \frac{b_{k }^{q/p}- b_{k +1 }^{q /p}}{(b_k-b_{k+1})^{q/p}} \to \infty ,~\mbox{as $n\to\infty$}. $$
This implies that
$\{B_n^{-1}\chi_{(b_{n+1},b_n]}\}$ in $L_{p,q}(0,1)$ is not equivalent to $ \{e^{\ell_q}_n\}$ in $\ell_q$.
\end{proof}

\begin{example}Theorem \ref{Flq} above shows that if a sequence of atoms $\{A_n\}$ is such that $\mu(A_n)=\frac{1}{2^n}$, $n\ge 1$, then the normalized basis of $\{\chi_{A_n}\}$ is equivalent to the unit vector basis of $\ell_q$.
\end{example}



Proposition \ref{Flq} gives a criterion for   the basis generated by $\chi_{A_n}$ to be  equivalent to the natural basis of $\ell_q$.
In Theorem \ref{4.19} below, we show that there exists $X$ of type $\ell_{p,q,0}(F)$, which is not isomorphic to $\ell_q$.

Recall the following important result due to Lindenstrauss and Zippin (see \cite[Theorem 2.b.9]{LT1} or \cite{LZ}).
\begin{lemma}\label{LZ}
$\ell_1$ has a unique unconditional basis up to equivalence.
\end{lemma}

 Lemma \ref{LZ} and  Proposition \ref{Flq} immediately imply that there exist Banach spaces of type $\ell_{p,1,0}(F)$ which are  not isomorphic to $\ell_1$.
 We prove below  a slightly stronger result.

\begin{corollary}\label{3.10}Let $p>1 $.
If $\limsup_{a\to 0}{\rm Card}( \{A_n :  \mu(A_n) =a \} ) =\infty$, then
$\ell_{p,1}(\{A_n \})\not\xhookrightarrow{c} \ell_1$.
\end{corollary}
\begin{proof}
If  $\limsup_{a\to 0}{\rm Card}( \{A_n :  \mu(A_n) =a \} ) =\infty$,
then there exists a decreasing sequence $\{a_k\}$ such that
${\rm Card}( \{A_n :  \mu(A_n) =a_k \} ) \to_k \infty $.
We denote by $B_k:=\{A_n :  \mu(A_n) =a_k  \} $, $k=1,2,\cdots$.
Assume that $\ell_{p,1}(\{A_n \})\xhookrightarrow{c} \ell_1$. Since $\ell_{p,1}(\{A_n\})$ is infinite-dimensional, it follows from   the fact that $\ell_1$ is a prime space \cite[p.57]{LT1}  that $\ell_{p,1}(\{A_n \})\approx  \ell_1$.
By the uniqueness of the unconditional basis of $\ell_1$, we obtain that
$\left\{ \frac{\chi_{A_n}}{\norm{A_n}_{p,1}}  \right\}_n \sim \left\{e^{\ell_1}_n\right \}_n$.
However,
$$ |B_k|=  \norm{  \sum_{i=1}^{|B_k|} e_i^{\ell_1}  }_1 \sim \norm{ \sum_{A_n\in B_k} \frac{\chi_{A_n}}{\norm{A_n}_{p,1}}  }_{p,1} = |B_k|^{1/p} , $$
which is impossible when $|B_k|\to \infty$.
\end{proof}


Under the same condition in Corollary \ref{3.10}, we ascertain  the lack of isomorphic embedding $\ell_{p,q}(\{A_n\})$ into $ \ell_{q}  $ when $1< p<\min\{2,q\}$ or $p>\max\{2,q\}$ or $p\ne q>2 $.
\begin{proposition}\label{4.14}
Let $\{A_n\}$ be such that
$\limsup_{a\to 0}{\rm Card}( \{A_n :  \mu(A_n) =a \} ) =\infty$.
If $1< p<\min\{2,q\}$ or $p>\max\{2,q\}$ or $p\ne q>2 $, then $$\ell_{p,q}(\{A_n\})\not\hookrightarrow \ell_{q} .$$
\end{proposition}
\begin{proof}
The condition $\limsup_{a\to 0}{\rm Card}( \{A_n :  \mu(A_n) =a \} ) =\infty$ implies that
there exists a sequence  of subsets $\{B^k_n\}_{1\le n\le n_k}$ of $\{A_n\}$
such that $\mu(B^k_n)= a_k$ and $a_k\to 0$ as $k\to \infty$,
where $n_k :={\rm Card}( \{A_n :  \mu(A_n) =a_k  \} )  $.

Noting  that
$\ell_{p,q}(\{B_n^k\}_{1\le n\le n_k})$ is isometric to $\ell_{p,q}^{n_k}$, it follows that
$ \ell_{p,q}(\{A_n\})$
contains $\ell_{p,q}^{n_k} $ uniformly.
By Corollary  \ref{concl},
$ \ell_{p,q}(\{A_n\})\not\hookrightarrow  \ell_q$.
\end{proof}



\begin{cor}\label{4.16}For any $1<p<\infty$, $1\le   q<\infty$ and $p\ne q$,
if $\limsup_{a\to 0}{\rm Card}( \{A_n :  \mu(A_n) =a \} ) =\infty$,
then
$\ell_{p,q}(\{A_n\}) \not  \xhookrightarrow{c} \ell_q $.
\end{cor}
\begin{proof}
By Propositions \ref{4.14}  and Corollary \ref{3.10}, the only remaining case is when $1<q <p\le 2$.

Assume that
  $\ell_{p,q}(\{A_n \}) \xhookrightarrow{c} \ell_q $.  Since $\ell_{p,q}(\{A_n\})$ is infinite-dimensional, it follows from the fact that $\ell_q $ is a prime space \cite[p.57]{LT1}  that $\ell_{p,q}(\{A_n \})\approx  \ell_q$.
Then,   letting $\frac{1}{p}+\frac{1}{p'}=1$ and  $\frac{1}{q}+\frac{1}{q'}=1$, we have that
$\ell_{p',q'}(\{A_n \}) $ can be identified with the dual of $\ell_{p,q}(\{A_n \}) $\cite[Chapter IV, Corollary 4.8]{Bennett_S}.
Hence,
$$\ell_{p',q'}(\{A_n \}) \approx  \ell_{q'},  $$
which is a contradiction with Proposition \ref{4.14}.
\end{proof}

A sequence $\{A_n\}$  of atoms in a measure space $(\Omega_1,\Sigma_1,\mu_1)$ is said to be equivalent to  a sequence $\{B_n\}$ of atoms  in a measure space $(\Omega_2,\Sigma_2,\mu_2)$ if $\{\mu_1(A_n)\} \sim\{ \mu_2(B_n)\}$.
By Corollary~\ref{4.16} and  Remark \ref{remark}, we obtain  the following result.

\begin{corollary}\label{4.17}
 For any $1<p<\infty$, $1\le   q<\infty$ and $p\ne q$,
if
  $\{A_n\}$ is equivalent to a set $\{B_n \}$ of atoms
such that
  $\limsup_{a\to 0}{\rm Card}( \{B_n :  \mu(B_n) =a \} )= \infty$, then
  $$\ell_{p,q}(\{A_n \}) \not \xhookrightarrow{c}\ell_q .$$
\end{corollary}

\begin{example}
Let $\{A_n\}$ be such that  $\mu(A_n) = \frac{1}{n}-\frac{1}{n+1}$ for all
$n\ge 1$.
We claim that $\ell_{p,q}(\{A_n\}) \not\xhookrightarrow {c} \ell_q$.
Indeed,
let $B_k:=\left\{A_n : \frac{1}{2^{k+1}}  \le  A_n < \frac{1}{2^k}\right \}=\left\{A_n :         2^k  < n(n+1 )   \le 2^{k+1} \right \}$.
Clearly, $|B_k|\to_k \infty$.
Let $\{C_k\}$ be a sequence of atoms such that   $\{C_k: \mu(C_k)=\frac{1}{2^n}\}=|B_n|$.
We obtain that
$$\{A_k\}\sim \{C_k\} .$$
By Corollary \ref{4.17}, we obtain that
$\ell_{p,q}(\{A_n\}) \not\xhookrightarrow {c} \ell_q$.
\end{example}

The following theorem provides a  necessary and sufficient condition for a space of type $\ell_{p,q,0}(F)$ to be isomorphic to $\ell_q$.
\begin{theorem}\label{4.19}Let $1<p<\infty$, $1\le   q<\infty$ and $p\ne q$.
Let $A_n$ be atoms such that $\sum_{n=1}^\infty \mu(A_n) =1 $ and $\mu(A_n)\ge \mu(A_{n+1})$ for any $n$.
Then, $X:=\ell_{p,q}(\{A_n\})$  satisfies one of the followings:
 \begin{enumerate}
               \item $X\approx \ell_q$ if  $\{A_k\}$ is equivalent to a set of atoms $\{B_n\}$ which satisfies   $\sup _k \frac{ \mu(B_{k+1 } )}{  \mu(B_k) }<1$.

  \item $X$ contains $\ell_{p,q}^{n}$, $n\ge 1$, uniformly  if $\{A_n \}$  is not  equivalent to a set of atoms $\{B_n\}$ which satisfies   $\sup _k \frac{ \mu(B_{k+1 } )}{  \mu(B_k) }=1$.
      In this case, $X\not \approx \ell_q$.
  \end{enumerate}
\end{theorem}
\begin{proof}
The first case follows from Proposition \ref{Flq}.

Assume $\{A_n \}$  is not  equivalent to a sequence of atoms $\{B_n\}$ which satisfies   $\sup _k \frac{ \mu(B_{k+1 } )}{  \mu(B_k) }=1$.
We claim that for any $a<1$,
$$\sup _k {\rm Card}  (a^{k+1} \le  \mu(A_n)<a^k  )  = \infty. $$
Otherwise, assume that there exists $a<1$ such that
$N:= \sup _k {\rm Card}  (a^{k+1} \le \mu(A_n)<a^k  )  <\infty.
$

When $k=0$,
 let
$N_0:= {\rm Card}(a  \le  \mu(A_n)<1  )  .
$
 For index $ 1\le i\le N_0$, let atoms $B_i$ be such that $\mu(B_i) =    a ^{\frac{i}{N_0}}  $.

When  $k=1$,  let
$N_1:=  {\rm Card} (a^{2} \le \mu(A_n)<a   )  .
$
For $ N_0+1 \le i\le N_0+ N_1$, let atoms $B_i$ be such that  $\mu(B_i) =    a ^{1+ \frac{i-N_0}{N_1}}  $.

Constructing  inductively, we obtain a sequence $\{B_n\}$ which is equivalent to $\{A_n\}$.
Moreover, $\frac{\mu(B_{n+1})}{\mu(B_{n})} <a^{\frac{1}{N}}<1 $, which is a contradiction with the assumption.
Hence,
$$\sup _k {\rm Card}  (a^{k+1} \le  \mu(A_n)<a^k  )  = \infty. $$
This implies that $X$ contains $\ell_{p,q}^{n}$ (see Remark \ref{remark}), $n\ge 1$, uniformly.
Applying Corollary \ref{4.17}, we complete the proof.
\end{proof}

We identify  below an important  subspace of $L_{p,q}(0,1)$.
\begin{theorem}\label{newsubspace}
Let $1<p<\infty$, $1\le q<   \infty$ and  $p\ne q$.
Let $\{n_i\}$ be an  increasing sequence.
Then, there exists a space $X$ of type $\ell_{p,q,0}(F)$ such that $$X\approx \left(\oplus_i  \ell_{p,q}^{n_i}\right)_{q}.$$
In particular, $L_{p,q}(0,1)$ contains a subspace isomorphic to  $\left(\oplus_i  \ell_{p,q}^{n_i}\right)_{q}$.
\end{theorem}
\begin{proof}
Let $b_n= \frac{1}{\prod_{1\le i\le n} (n_i+1 )}$.
Let $A_k^i$, $1\le k\le n_i$, be mutually disjoint left-closed and right-open intervals such that $\mu(A_k^i)= b_i$ and $\cup_{1\le k\le n_i} A^i_k= [b_i , b_{i-1})$, $\cup_{1\le k\le n_1} A^1_k= [b_1 , 1)$.
In particular,
 $$b_{i}=b_{i+1} n_{i+1}+b_{i+1}, ~  \sum_{1\le k\le n_i } \mu(A_k^i ) = b_i  n_i   $$
 and
 $$ 1= b_1 (n_1+1) =b_1n_1+b_1  = b_1 n_1 +b_2  n_2+b_2=\cdots =\sum_{i\ge 1}b_in_i
 =
 \sum_{1\le k\le n_i, ~i\ge 1} \mu(A_k^i ). $$
Define $X:=\ell_{p,q}(\{A_n\})$.

Let $p>q$ (the case for $q>p$ follows from a similar proof).
For any sequence  $a:=(a_1^1, \cdots, a_{n_1}^1, a_{1}^2,\cdots, a_{n_2}^2,\cdots )$ of finite non-zero numbers (without loss of generality, we may assume that for every $i$, $a^i:=(a_{1}^i, \cdots, a_{n_i}^i)$ is non-increasing), by the Hardy--Littlewood inequality,
letting
$\displaystyle x: =  \sum_{  1\le k\le n_i, i\ge 1}    a_k^i \frac{\chi_{A_k^i}}{\norm{\chi_{A_k^i}}_{p,q}}  $,
we obtain that
\begin{align*}
   \int_0^1   x  ^*(s) ^qds^{q/p}   &\geq \int_0^1 |x(s)|^qds^{q/p}=
 \int_{b_1}^{1}    \sum_{  1\le k\le n_1 }   ( a_k^1)^q  \frac{\chi_{A_k^1}}{\norm{\chi_{A_k^1}}_{p,q}^q }        ds^{q/p} +
 \sum_{i\ge 1 }  \int_{ b_{i+1} }^{ b_i    }  |x(s)|^qds^{q/p} \\
 &=
 \int_{1}^{n_1 +1 }
  \sum_{1\le k\le n_1} |a_{k}^{ 1 }|^q \chi_{[  k,k+1  )} ds^{q/p} +
 \sum_{i\ge 2 }  \int_{ 1   }^{n_i+1  }\sum _{1\le k\le n_i}   |a_k^i |^q \chi_{[k,k+1)}   ds^{q/p}.
   \end{align*}
Note that $\{(k+1 )^{q/p}  - k^{q/p}\}_{k\ge 1} \sim \{k  ^{q/p}  - (k-1)^{q/p}\}_{k\ge 1}  $,
we obtain that
$$
    \int_0^1  x^*(s)^qds^{q/p}  \gtrsim  \norm{ x   }_{(\oplus \ell_{p,q}^{n_j})_{\ell_q}}^q .$$
On the other hand, by the $q$-convexity of $L_{p,q}$ (see e.g. \cite[Eq.(17)]{Dilworth}), we obtain that
$$ \norm{x}_X  \le \left( \sum_{i\ge 1}  \norm{ (a^i)}_{\ell_{p,q}^{n_i}} ^q \right) ^{1/q} .$$
Hence,
the unconditional basis of $X$ is equivalent to the unconditional basis of $(\oplus \ell_{p,q}^{n_j})_{q}$, which completes the proof.
\end{proof}

Now, we show below that for any increasing sequence  $\{n_i\}_{i=1}^\infty $, all spaces $\left(\oplus_i  \ell_{p,q}^{n_i}\right)_{q}$ are isomorphic to each other.
\begin{proposition}Let $1<p<\infty$, $1\le   q<\infty$ and $p\ne q$.
Let $\{n_i\}$ be an  increasing sequence.
The space $\left(\oplus_i  \ell_{p,q}^{n_i}\right)_{q}$ does not rely on the choice of
the sequence $\{n_i\}$ (up to an isomorphism).
\end{proposition}
\begin{proof}
Let $\{k_j\}$ be   a sequence  of natural numbers, which contains  infinitely  many entries equal to $n$ for every $n\ge 1$.


Let $U:=  (\oplus_{j} \ell_{p,q}^{k_j} )_q$. 
It is clear that
\begin{align}\label{U2U}U\oplus U=U.
\end{align} 
Let $X:  = \left(\oplus_i  \ell_{p,q}^{n_i}\right)_{q}$.
In particular, we have
\begin{align}\label{XUU}
X\oplus U\approx U.
\end{align}
Note that $U$ is isomorphic to a complemented subspace of $X$.
Therefore, there exists a Banach space $Y$ such that $X\approx U\oplus Y$. Hence, we have
$$U\stackrel{\eqref{XUU}}{\approx } X\oplus U \approx (U\oplus Y)\oplus U\stackrel{\eqref{U2U}}{\approx}     U\oplus Y \approx X. $$
This completes the proof.
\end{proof}

\begin{rem}
We know that $\ell_q\oplus \ell_{p,q}\approx \ell_{p,q}$.
It is interesting to note that there exist some other  spaces of type $\ell_{p,q,0}(F)$ which  can be `absorbed' by $\ell_{p,q}$.

Assume that atoms $A_n$ satisfy that
$\mu(A_n )=\frac{1}{n^2}$.
Let  atoms $B_n$ satisfy that
$\mu(B_n )=\frac{1}{(2n)^2}$ and
 atoms $C_n$ satisfy that
$\mu(C_n )=\frac{1}{(2n-1) ^2}$.

Since
 $$ \frac{1}{n^2 }  \ge
  \frac{1}{(2n- 1)^2} \ge \frac{1}{(2n)^2}
= \frac{1}{ 4    n ^2 } $$
for every $n$,
it follows that (see Remark \ref{remark})
$$\ell_{p,q}(\{A_n \})\approx\ell_{p,q}(\{B_n \})\approx \ell_{p,q}(\{C_n \})  .$$
Note that $\ell_{p,q}(\{A_n \})\approx\ell_{p,q}(\{B_n \})\oplus \ell_{p,q}(\{C_n \})  $.
We obtain that $\ell_{p,q}(\{A_n \})\approx \ell_{p,q}(\{A_n \})\oplus \ell_{p,q}(\{A_n \})$.
In particular, by Pe{\l}czy\'{n}ski decomposition method\cite[Theorem 2.2.3]{AK}, we obtain that  $\ell_{p,q}\approx \ell_{p,q}\oplus \ell_{p,q}(\{A_n \})$.
\end{rem}

\begin{rem}
It is clear that
when the atoms $A_n$ such that    $\mu(A_n)=\frac{1}{n^\alpha}$, $\alpha>1$,
$(\oplus _{n\ge 1}\ell_{p,q}^n)\xhookrightarrow{c}\ell_{p,q}(\{A_n\})$.
Indeed, let $\{B_n\}$ be a sequence of atoms such that
$\mu(B_1)=\frac{1}{2}$, $\mu(B_2)=\mu(B_3)= \frac{1}{3}\mu(B_1)=\frac16$, $\mu(B_4)=\mu(B_5)=\mu(B_6)=\frac{1}{4}\mu(B_2)=\frac{1}{24}$, $\cdots$.
In particular, there are $n$'s many $B_k$ such that $\mu(B_k)=\frac{1}{\prod _{1\le i\le n}(i+1)}$.
Note that $\ell_{p,q}(\{B_n\})\approx (\oplus_{i=1}^\infty \ell_{p,q}^i)_q$ (see the proof of Theorem \ref{newsubspace}).
Let $n_k=[(\prod _{1\le i\le k}(i+1))^{1/\alpha}]-1$, where $[(\prod _{1\le i\le k}(i+1))^{1/\alpha}] $ is  the smallest integer not less than $(\prod _{1\le i\le k}(i+1))^{1/\alpha}$.
Note  that
$(2n_k )^\alpha \ge (n_k +1)^\alpha \ge \prod _{1\le i\le k}(i+1) \ge c(\alpha)  k ^\alpha $ (here, $c(\alpha)$ is a constant depending on $\alpha$ only).
In particular, $\frac{2}{c(\alpha)^{\frac{1}{\alpha}}}  n_k  \ge
  k $.
We obtain  that
\begin{align*}
\frac{1}{(1 +\frac{2}{c(\alpha)^{\frac{1}{\alpha}}})^\alpha \prod _{1\le i\le k}(i+1)}&  \le    \frac{1}{(1 +\frac{2}{c(\alpha)^{\frac{1}{\alpha}}})^\alpha ( n_k    )^\alpha}   =
 \frac{1}{  (  n_k  +\frac{2}{c(\alpha)^{\frac{1}{\alpha}}} n_k    )^\alpha}   \\
&\le   \frac{1}{(n_k  +k -1 )^\alpha}   \le  \frac{1}{\prod _{1\le i\le k}(i+1)}  .
\end{align*}
This implies that
$\cup _{k\ge 1}\{A_{n_k}, A_{n_k+1},\cdots,A_{n_k+k-1} \} $ is equivalent to $\{B_k \}_{k\ge 1}$,
and therefore, $(\oplus _{n\ge 1}\ell_{p,q}^n)\xhookrightarrow{c}\ell_{p,q}(\{A_n\})$.
\end{rem}
\begin{quest}
It will be interesting to classify
all spaces $\ell_{p,q}(\{A_n\})$ of this type (up to an isomorphism).
That is, letting atoms $A_n$ be  such that    $\mu(A_n)=\frac{1}{n^\alpha}$,  $\alpha>1$, and
 atoms $B_n$ be such that    $\mu(B_n)=\frac{1}{n^\beta}$,
 $\beta>1$, $\alpha\ne \beta$,
does
$$  \ell_{p,q}(\{A_n\})  \approx  \ell_{p,q}(\{B_n\}) ? $$
 One may ask a more general question: are there    infinitely many isomorphism types of $\ell_{p,q,0}(F)$ spaces?
\end{quest}

\subsection{The $\ell_{p,q,\infty}$ case}
Without loss of generality, we always assume in this subsection that
$\sum _{n=1}^{\infty }\mu (A_{n})=\infty $ and
$\mu (A_{n} )\le \mu (A_{n+1})$ for every $n$. Let
$b_{n} := \sum _{ k\le n}\mu (A_{k} ) $. Then,
$\ell _{p,q}(\{A_{n}\})$ can be identified as the subspace of
$L_{p,q}(0,\infty )$ generated by intervals $[ b_{n}, b_{n+1})$. The proofs
of results in this subsection are very similar to those for
$\ell _{p,q,0}(F)$. Therefore, some of the proofs are omitted.

\begin{proposition}%
\label{noti}
Let $1<p<\infty $, $1\le q<\infty $ and $p\ne q$. If $X$ is of type
$\ell _{p,q,\infty }$, then
\begin{equation*}
\ell _{p,q}\not \hookrightarrow X.
\end{equation*}
In particular, $U_{p,q}\not \hookrightarrow X$.
\end{proposition}
\begin{proof}
Assume by contradiction that there exists an isomorphism
$T:\ell _{p,q}\hookrightarrow X$. Note that
$e^{\ell _{p,q}}_{k}\hookrightarrow _{k} 0$ weakly in $\ell _{p,q}$. Hence,
$T(e^{\ell _{p,q}}_{k})\to _{k} 0$ weakly in $X$. By the Bessaga--Pe{\l }czy\'{n}ski
Selection Principle \cite[Proposition 1.3.10]{AK}, passing to a subsequence
if necessary, we obtain that $\{e^{\ell _{p,q}}_{k}\}$ in
$\ell _{p,q}$ is equivalent to a block basic sequence of the unconditional
basis
$\left \{
\frac{\chi _{A_{k}}}{\left \lVert \chi _{A_{k} }\right \rVert }
\right \}  _{k} $ in $X$. Since $\mu (A_{n})\to \infty $, it follows that
$\{T(e^{\ell _{p,q}}_{k})\}_{k} $ in $X$ goes to $0$ in the uniform norm.
Hence, there exists a subsequence of
$\{T(e^{\ell _{p,q}}_{k})\}_{k} $ equivalent to
$\{e^{\ell _{q}}_{k}\}$. This implies that $\{e^{\ell _{p,q}}_{k}\}$ in
$\ell _{p,q}$ is equivalent to $\{e^{\ell _{q}}_{k}\}$ in
$\ell _{q}$, which is impossible.
\end{proof}

The following proposition can be established along the same line as in
the proof of  Proposition~\ref{Flq}  and therefore its proof is omitted.
Note that the following result holds true for $0<p,q<\infty $,
$p\ne q$.

\begin{proposition}%
\label{infiq}
Let $1<p<\infty $, $1\le q<\infty $, $p\ne q$. Let
$0= b_{1}< \cdots <b_{n} <b_{n+1}< \cdots \to \infty $. Let
$B_{n} = \left \lVert \chi _{(b_{n },b_{n+1}]}\right \rVert _{p,q} = (b_{n+1}-b_{n})^{1/p}$.
Then, the normalized basis $\{B_{n}^{-1}\chi _{(b_{n },b_{n+1}]} \}$ in
$L_{p,q}(0,\infty )$ is equivalent to $ \{e^{\ell _{q}}_{n}\}$ if
$\sup _{n} \frac{b_{n}}{b_{n+1}}<1$. In particular,
$\ell _{p,q}(\{A_{n}\})\approx \ell _{q}$ if
$\sup _{n}
\frac{\sum _{k=1 }^{n} \mu (A_{k}) }{\sum _{k=1}^{n+1} \mu (A_{k})}<1$,
where $\mu (A_{n})=b_{n+1}-b_{n}$.

On the other hand, if $\lim _{n\to \infty } \frac{b_{n+1}}{b_{n}}=1 $, then
$\{B_{n}^{-1}\chi _{(b_{n+1},b_{n}]}\}$ in $L_{p,q}(0,1)$ is not equivalent
to $ \{e^{\ell _{q}}_{n}\}$.
\end{proposition}
The following results follow from the same argument in the proof of  Corollaries~\ref{3.10}, \ref{4.16} and \ref{4.17}, Proposition~\ref{4.14}.
%
\begin{corollary}
Let $1<p<\infty $, $1\le q<\infty $ and $p\ne q$. If
$\limsup _{a\to \infty }{\mathrm{Card}}( \{A_{n} : \mu (A_{n}) =a \} ) =\infty $,
then $\ell _{p,q}(\{A_{n} \})\not \xhookrightarrow{c} \ell _{1}$.
\end{corollary}

\begin{proposition}
Let $\{A_{n}\}$ be such that
$\limsup _{a\to \infty }{\mathrm{Card}}( \{A_{n} : \mu (A_{n}) =a \} ) =\infty $.
If $1< p<\min \{2,q\}$ or $p>\max \{2,q\}$ or $p\ne q>2 $, then\nopagebreak
\begin{equation*}
\ell _{p,q}(\{A_{n}\})\not \hookrightarrow \ell _{q} .
\end{equation*}
\end{proposition}

\begin{cor}
For any $1<p<\infty $, $1\le q<\infty $ and $p\ne q$, if
$\limsup _{a\to \infty }{\mathrm{Card}}( \{A_{n} : \mu (A_{n}) =a \} ) =\infty $,
then $\ell _{p,q}(\{A_{n}\}) \not \xhookrightarrow{c} \ell _{q} $.
\end{cor}

\begin{corollary}%
\label{4.28}
For any $1<p<\infty $, $1 \le q<\infty $ and $p\ne q$, if
$\{A_{n}\}$ is equivalent to a set $\{B_{n} \}$ of atoms such that
$\limsup _{a\to  \infty}{\mathrm{Card}}( \{B_{n} : \mu (B_{n}) =a \} )= \infty $,
then
\begin{equation*}
\ell _{p,q}(\{A_{n} \}) \not \xhookrightarrow{c}\ell _{q} .
\end{equation*}
\end{corollary}

\begin{example}
Let $\{A_{n}\}$ be such that $\mu (A_{n}) = n $ for all $n\ge 1$. Let
$B_{k}:=\left \{  A_{n} : 2^{k } \le \mu(A_{n}) < 2^{k+1}\right \}  =\left \{  A_{n}
: 2^{k} < n \le 2^{k+1} \right \}  $. Clearly,
$|B_{k}|\to _{k} \infty $ and $\{A_{k}\}\sim \{C_{k}\}$, where there are
$|B_{n}|$'s many $C_{k}$ such that $\mu (C_{k})=2^{n}$. Hence, we obtain
that $\ell _{p,q}(\{A_{n}\}) \not \xhookrightarrow{c} \ell _{q}$.
\end{example}

An argument similar to that in  Theorem~\ref{newsubspace}  yields the following
result.
%
\begin{theorem}%
\label{4.30}%
Let $1<p<\infty $, $1\le q<\infty $ and $p\ne q$. Let $\{n_{i}\}$ be an
increasing sequence. Then, there exists a space $X$ of type
$\ell _{p,q,\infty }$ such that
\begin{equation*}
X\approx \left (\oplus _{i} \ell _{p,q}^{n_{i}}\right )_{q}.
\end{equation*}
In particular, $X$ is isomorphic to a space of type
$\ell _{p,q,0}(F)$ and $L_{p,q}(0,1)$ contains $X$.
\end{theorem}

Recall that  Proposition~\ref{0c}  shows that all spaces
$\ell _{p,q}(\{A_{n}\})$ can be absorbed by $U_{p,q}$. The following proposition
shows that all $X$ of type $\ell _{p,q,\infty }$ can be absorbed by
$\ell _{p,q}$, which follows in fact from
\cite[Proposition 3.a.5]{LT1}. Note that
$\ell _{q}\oplus \ell _{p,q}\approx \ell _{p,q}$ is a direct consequence
of \cite[Theorem 2.2.3]{AK}.
%
\begin{proposition}%
\label{lpqabsorb}
For any $1<p<\infty $, $1\le q<\infty $ and $p\ne q$, if
$X:=\ell _{p,q}(\{A_{n}\})$ is of type $\ell _{p,q,\infty }$, then
\begin{equation*}
X\oplus \ell _{p,q}\approx \ell _{p,q}.
\end{equation*}
In particular,
$\left (\oplus _{i=1}^{\infty }\ell _{p,q}^{ i}\right )_{q}\oplus \ell _{p,q}
\approx \ell _{p,q}$.
\end{proposition}

\begin{quest}
 Are there    infinitely many isomorphism types of $\ell_{p,q,\infty }$ spaces?
\end{quest}

\section{Isomorphic classification of $L_{p,q}$: general $\sigma$-finite case with a non-trivial atomless part}\label{general}
In this section, we consider $L_{p,q}$-spaces on a general
$\sigma$-finite  measure space with a non-trivial atomless part.

We state  below results of  isomorphic classification of $L_{p,q}$-spaces obtained in this section and supply the proof which is also based on a number of results given later in this subsection.
 \begin{theorem}\label{2M} Let $1<p<\infty$, $1\le q<\infty$, $p\ne q$.
 Let $(\Omega,\Sigma,\mu)$
 be a $\sigma$-finite measure space which is not purely atomic.
 Then, $L_{p,q}(\Omega)$ is of  one of the following types (up to an isomorphism)
 $$L_{p,q}(0,1), ~L_{p,q}(0,\infty), ~L_{p,q}(0,1)\oplus U_{p,q} ,  ~L_{p,q}(0,1)\oplus \ell_{p,q },  ~L_{p,q}(0,1)\oplus \ell_{p,q,\infty }. $$
 \end{theorem}
 \begin{proof}
By Corollary \ref{primary}, $L_{p,q}(0,\infty) \oplus L_{p,q}(\Omega_1)\approx L_{p,q}(0,\infty)$ for any atomic $\sigma_1$-finite measure space $\Omega_1$.
Hence, if the atomless part of the measure  space $\Omega$ has infinite measure, then $L_{p,q}(\Omega)\approx L_{p,q}(0,\infty)$.

If the atomless part of the measure $\Omega$ has finite measure, then, by Corollary \ref{cortype},
all logically possible types of $L_{p,q}(\Omega)$ are the following (up to an isomorphism):
\begin{align*}L_{p,q}(0,1), ~L_{p,q}(0,1)\oplus \ell_{p,q},~ L_{p,q}(0,1)\oplus  \ell_{p,q,\infty } ,~L_{p,q}(0,1)\oplus  \ell_{p,q,0}(F),  \\
L_{p,q}(0,1)\oplus U_{p,q},~  L_{p,q}(0,1)\oplus  \ell_{p,q,0}(F)\oplus \ell_{p,q},~L_{p,q}(0,1)\oplus  \ell_{p,q,0}(F)\oplus \ell_{p,q,\infty } .
\end{align*}
By Proposition \ref{5.5}, we obtain that
$$L_{p,q}(0,1)\oplus \ell_{p,q,0}(F)\approx L_{p,q}(0,1) .$$
Hence, if the atomless part of the measure $\Omega$ has finite measure, then
  $L_{p,q}(\Omega)$ is of one of the following types (up to an isomorphism):
\begin{align*}L_{p,q}(0,1), ~L_{p,q}(0,1)\oplus \ell_{p,q},~ L_{p,q}(0,1)\oplus  \ell_{p,q,\infty } ,  ~L_{p,q}(0,1)\oplus U_{p,q}.
\end{align*}
This completes the proof.
\end{proof}

Indeed, in this section, we obtain  results for isomorphic embeddings between different types of $L_{p,q}$, which are collected in the following table. We supply below all the references to every line in this table  (from left to the right).
Note that whenever $X\hto Y$ in the following table, $X$ is indeed isomorphic to a complemented subspace of $Y$.

The assertions in the 1st line are well-known facts\cite{Dilworth}.

The assertions  in the 2nd line
follow  from Proposition \ref{noti}, Proposition \ref{noti},   Proposition \ref{0c},  \cite{SS2018, KS}, the  discussion in Section \ref{prel}, a trivial observation,  Proposition \ref{0c}, respectively.

The assertions  in the 3rd   line follow  from from Corollary \ref{concl} and Remark \ref{rem:s}, Proposition \ref{lpqabsorb}, Proposition \ref{0c},   Theorem \ref{newsubspace},  the  discussion in Section \ref{prel},  Theorem \ref{newsubspace},  Theorem \ref{newsubspace}, respectively.

The assertions in the 4th  line follow  from Proposition \ref{4.8},  Proposition \ref{4.8}, Proposition \ref{noti},   Proposition \ref{0c} and \cite{SS2018,KS},  the  discussion in Section \ref{prel}, Theorem \ref{5.10}, a trivial observation, respectively.

 The assertions in the 5th line follow  from
Theorem \ref{ifonlyif2}, Theorem \ref{ifonlyif2},  Theorem \ref{ifonlyif2}, Theorem \ref{ifonlyif2}, the  discussion in Section \ref{prel}, a trivial observation, a trivial observation, respectively.

 The assertions in the 6th line  follow  from
Theorem \ref{ifonlyif2}, Theorem \ref{ifonlyif2},  Theorem \ref{ifonlyif2},  Theorem \ref{ifonlyif2},  Corollary \ref{infinityinto}, Corollary \ref{infinityinto}, Corollary \ref{infinityinto}, respectively.

 The assertions in the 7th line follow  from Theorem \ref{ifonlyif2}, Theorem \ref{ifonlyif2},   Theorem \ref{ifonlyif2}, Theorem \ref{ifonlyif2}, \cite{SS2018,KS}, the  discussion in Section \ref{prel}, Proposition \ref{0c}, respectively.

 The assertions in the 8th line  follow  from Theorem \ref{ifonlyif2}, Theorem \ref{ifonlyif2},   Theorem \ref{ifonlyif2},  Theorem \ref{ifonlyif2}, Proposition \ref{0c} and
\cite{SS2018,KS},  the  discussion in Section \ref{prel}, Theorem \ref{5.10},   respectively.\\

 \begin{table}[ht!]
\begin{center}
\begin{tabular}{|>{\centering} m{1.8cm}  | >{\centering}m{2cm} | >{\centering}m{0.4cm} |>{\centering} m{1.7cm} |>{\centering} m{0.4cm} |>{\centering} m{1.2cm} |>{\centering} m{1.4cm}   |c  |c |   }

  \hline
\diagbox[width=6em,height=4em]{$X$}{$Y$}& $\ell_q$ & $\ell_{p,q}$& $(\oplus_{n=1}^\infty \ell_{p,q}^n )_q$& $U_{p,q}$&  $L_{p,q}(0,1)$ & $L_{p,q}(0,\infty )$ & \makecell{ $ L_{p,q}(0,1)$\\$\oplus$\\$\ell_{p,q}$} &   \makecell{ $  L_{p,q}(0,1)$\\$\oplus$\\$U_{p,q}$ } 
   \\
  \hline
 $\ell_{q}$  & $-$  & Yes  & Yes & Yes & Yes & Yes & Yes & Yes
   \\\hline
    $\ell_{p,q}$  &No & $-$  & No & Yes & No & Yes & Yes & Yes 
   \\\hline
$(\oplus_{n=1}^\infty \ell_{p,q}^n )_q$ &Sometimes & Yes &- & Yes & Yes & Yes & Yes & Yes
\\\hline
    $U_{p,q}$ &No & No & No  & $-$ & No & Yes & No & Yes 
   \\\hline
  $L_{p,q}(0,1)$ &No & No& No &  No& $-$ & Yes & Yes & Yes 
     \\\hline
  $L_{p,q}(0,\infty )$  &No  &No& No &No &  No &$-$ & No & No  
     \\\hline
\makecell{ $ L_{p,q}(0,1)$\\$\oplus$\\$\ell_{p,q}$}&No  & No& No & No& No & Yes & $-$  &  Yes 
 \\\hline
\makecell{ $  L_{p,q}(0,1)$\\$\oplus$\\$U_{p,q} $ } &No  & No& No &  No&  No & Yes & No     & $-$ 
 \\\hline

\end{tabular}
\captionof{table}{Does $X$ embed into $Y$?}\label{t2}
\end{center}
\end{table}



 When $p>q=2$, every subspace of  $L_{p,q}(0,1)$ contains $\ell_2$, which presents additional technical obstacles in
 the proof for Theorem \ref{ifonlyif2}.
To resolve this case, we need the following proposition, which   can be obtained by the same construction in \cite[Lemma 21]{CD} and \cite[Proposition 1]{Dilworth}.
The result remains valid  for   weighted Lorentz sequence spaces \cite[Proposition 4.e.3]{LT1}.

\begin{proposition}\label{p2folklore}
Let $p>2$.
Let $\{x_i\}_i\subset  L_{p,2}(0,\infty)$ be a sequence of  normalized  disjointly supported elements  which converges to $0$ in measure.
Then, for any $\varepsilon\in (0,1)$, there exists an integer  $n$ such that for any $i\ge n$, we have
$$(1-\varepsilon)\sqrt{\lambda_1^2+\lambda_2^2} \le  \norm{\lambda_1x_1  +\lambda_2x_n}_{p,2} \le \sqrt{\lambda_1^2+\lambda_2^2}, ~\forall (\lambda_1,\lambda_2)\in \mathbb{R}^2    .$$
\end{proposition}

Recall that $L_{p}(0,1)\hto \ell_{p}$ if and only if $p=2$    \cite[Ch. XII, Theorem 9]{Banach}.
This  result was extended to the setting of Orlicz spaces by
Lindentrauss and Tzafriri, see
\cite[Theorem 2.c.14]{LT2} and \cite{LT3}.
Theorem  \ref{ifonlyif2} below is a Lorentz space counterpart of these two theorems (see \cite[Theorem 6.2]{AHS} for a related result).

Before proceeding to  Theorem  \ref{ifonlyif2} below, recall the definition of dilation operators\cite{KPS,LT2}.
If $\tau>0,$ the dilation operator $\sigma_{\tau}$ is defined by setting $\sigma_{\tau}f(s)=f(s/{\tau}),$ $s>0,$ in the case of the semi-axis. For a function on  the interval $(0,1),$ the operator $\sigma_{\tau}$ is defined by
$$
\sigma_{\tau}f(s)=
\begin{cases}
f(s/\tau),& s\leq\min\{1,\tau\}, \\
0,& \mbox{otherwise}.
\end{cases}
$$

\begin{theorem}\label{ifonlyif2}Let $1<p<\infty$ and $1 \le q\le \infty $ and let     $\{A_n\}_{n=1}^\infty$ be a sequence of atoms with finite measures.
Then,
$$ L_{p,q}(0,1)   \hookrightarrow \ell_{p,q}(\{A_n \})
$$
if and only if $p=q=2$.
\end{theorem}
\begin{proof}
Without loss of generality, we may assume that $\ell_{p,q}(\{A_n \})= U_{p,q}$.
It is well-known that if $p=q$, then
$\ell_{p,q}({A_n })$ is isomorphic to $\ell_q$.
By  \cite[Theorem 1]{S01}, we have  $L_p(0,1)\hto \ell_p$ if and only if $p=2$.
Hence, we may always assume   that $p\ne q$.


 (1) Let $1<p<2$. Assume by contradiction that
$ L_{p,q}(0,1) \hookrightarrow U _{p,q} $. Then,
$\ell _{r}\hookrightarrow L_{p,q}(0,1) \hookrightarrow U_{p,q} $ for any
$p<r\le 2$~\cite[Theorem 11]{Dilworth}. Applying Bessaga--Pe{\l }czy\'{n}ski
selection principle \cite[Proposition 1.3.10]{AK} and
\cite[Theorem 5]{Dilworth}, we obtain that for any $r\in (p,2]$,
$\ell _{r}$ contains $\ell _{q}$, which is impossible when $r\ne q$~\cite[Corollary 2.1.6]{AK}.

(2) Let $q\ne 2$.
Assume by contradiction that $ L_{p,q}(0,1)  \hookrightarrow U_{p,q}   . $
Then, $\ell_2\hookrightarrow L_{p,q}(0,1) \hookrightarrow U_{p,q}  $
 \cite[Theorem 11]{Dilworth}.
Let $T $ be an isomorphism from $\ell_2$ into $U _{p,q} $.
Since the unit vector basis of $\ell_2 $ is weakly null, it follows the Bessaga--Pe{\l}czy\'{n}ski selection principle   \cite[Proposition 1.3.10]{AK} that, passing to a subsequence if necessary, we may assume that  $T(e^{\ell_2}_n)$ are disjoint supported in $U_{p,q }$.
Hence, by \cite[Theorem 5]{Dilworth}, $\ell_2$ contains a subspace isomorphic to $\ell_q$,
which is impossible\cite[Corollary 2.1.6]{AK}.

 (3) Let $p>q=2$.

Let $T$ be an isomorphism from $L_{p,2}(0,1)$ into $U_{p,q} $.
For the sake of convenience, we may assume, in addition,  that the image of the Haar basis of $L_{p,2}(0,1)$ is a block basis of the unit vector basis of $U_{p,q} $\cite[Theorem 2.c.8]{LT2}.

We defined
$ r_{n,k,j}:= r_n \chi_{[\frac{j}{2^k},\frac{j+1}{2^k})} $ for any $n\ge 1$, $0\le j \le 2^k-1$, where $r_n$ is the Rademacher functions.
Recall that $\{r_n\}$ in $L_{p,q}(0,1)$ is equivalent to  the natural basis of $\ell_2$\cite{RS}.

 Let $n_0$ be an integer such that $x_0: =\frac{\sum_{i=1}^{n_0}  r_{i,k,0 }}{\sqrt{n_0}}$ satisfies
$\norm{   \sigma_{2^k } (x_0^*)  - \xi_0 ^*}_{p,2} <\frac{\varepsilon}{2^k} $, where $\xi_0 $ is a normal distribution (see Theorem \ref{convertoG}).
By \cite[Proposition 2.3]{CD}, there exists an increasing sequence of integers $(i_k)$ and a subsequence $( r'_{i,k,1 })$ of $( r_{i,k,1 })$ such that
$$F_n:=T\left(\frac{\sum_{i=i_n+1}^{i_{n+1} }  r'_{i,k,1 }}{\sqrt{i_{n+1}-i_n  }}\right )$$
converges to $0$ in measure.
Hence, for any $\varepsilon_1$, by Proposition \ref{p2folklore} and  Theorem \ref{convertoG},  there exists an integer  $n_1  $   such that
$x_1:=\frac{\sum_{i=i_{n_1}+1}^{i_{n_1+1 } }  r'_{i,k,1 }}{\sqrt{i_{n_1+1}-i_{n_1}   }} $ satisfies that
\begin{align*}
\norm{ Tx_0 +  Tx_1  }_{p,2}=
\norm{\norm{Tx_0}_{p,2} \frac{Tx_0}{\norm{Tx_0}_{p,2}}+ \norm{Tx_1}_{p,2} \frac{Tx_1}{\norm{Tx_1}_{p,2} } }_{p,2} &\ge (1+\varepsilon_1)^{-1} \norm{ \left(\norm{Tx_0}_{p,2}, \norm{Tx_1}_{p,2}\right)  }_{2}\\
& = (1+\varepsilon_1)^{-1}  \sqrt{\norm{Tx_0}_{p,2} ^2 +\norm{Tx_1}_{p,2}^2 }.
\end{align*}
 and $$\norm{   \sigma_{2^k } (x_1^*)  - \xi_1^* } _{p,2}<\frac{\varepsilon}{2^k}, $$
where $\xi_1 $ is a normal distribution.

 For any $\varepsilon_i$, $1\le i\le 2^k-1$,
construct inductively as above, we obtain  $x_i$, $0\le i\le 2^k-1$, such that
\begin{align}\label{Tiii}
 \norm{T(\sum_{i=0}^{2^k-1}x_i )}_{p,2}
&\ge    \frac{ \sqrt{ \sum_{i=0}^{2^k-1}
\norm{T(x_i)}_{p,2}^2   } }{\prod_{i=1}^{2^k-1}(1+\varepsilon_i)}  \ge \frac{ 2^{k/2 } \norm{T^{-1}}^{-1 }\min_{0\le i\le 2^k-1} \norm{ x_i}_{p,2} }{\prod_{i=1}^{2^k-1}(1+\varepsilon_i)}
\end{align}
and
\begin{align}\label{5.1}
 \norm{   \sigma_{2^k } (x_i^*)  - \xi_i ^* } _{p,2}<\frac{\varepsilon}{2^k},
\end{align}
where $\xi_i  $ is a normal distribution.

By Ryff's theorem \cite[Theorem 7.5]{Bennett_S}, we obtain that there exists a measure preserving transformation $\tau_i$ on $(0,1)$
such that
\begin{align}\label{5.33}
\norm{\sigma_{2^k}(x_i  )   -\xi_i^* \circ \tau_i}_E = \norm{\sigma_{2^k}(x_i^* ) -\xi_i^* }_E\stackrel{\eqref{5.1}}{\le}  \frac{\varepsilon}{2^k} .
\end{align}
Clearly, $\xi_i^* \circ \tau_i$ is also a normal distribution  and   $\xi_i^* \circ \tau_i$, $0\le i\le 2^k-1$,  are disjointly supported  because $x_i$ and  $\tau_i$ are disjointly supported, respectively\cite[Theorem 7.5]{Bennett_S}.
Hence,
$$\norm{  \xi_0 }_{p,2}+\frac{1}{2^{kp}} \varepsilon =  \norm{\sigma_{1/2^k }(\sum_{i =0}^{2^k -1} \xi_i^* \circ \tau_i)}_{p,2}+\frac{1}{2^{kp}}\varepsilon \stackrel{\eqref{5.33}}{\ge}   \norm{\sum_{i=0}^{2^k-1 }  x_i }_{p,2}  \ge   \norm{T }^{-1}   \norm{T(\sum_{i=0}^{2^k-1}x_i )}_{p,2}  .$$
By \eqref{Tiii}, we have
\begin{align*}
\norm{  \xi_0 }_{p,2}+\frac{1}{2^{kp}} \varepsilon
&~\ge \norm{T }^{-1} \frac{ 2^{\frac{k}{2} } \norm{T^{-1}}^{-1 }\min_{0\le i\le 2^k-1} \norm{ x_i}_{p,2} }{\prod_{i=1}^{2^k-1}(1+\varepsilon_i)}  \\
& ~ \ge \norm{T }^{-1} \frac{ 2^{\frac{k}{2} -\frac{k}{p}} \norm{T^{-1}}^{-1 }\min_{0\le i\le 2^k-1} \norm{ \sigma_{2^k} x_i}_{p,2} }{\prod_{i=1}^{2^k-1}(1+\varepsilon_i)}\\
&\stackrel{\eqref{5.33}}{\ge}
\norm{T }^{-1} \frac{ 2^{\frac{k}{2} -\frac{k}{p}} \norm{T^{-1}}^{-1 } (\norm{ \xi_0}_{p,2}-\frac{\varepsilon}{2^k}) }{\prod_{i=1}^{2^k-1}(1+\varepsilon_i)}
\end{align*}
Since $k$ can be taken arbitrarily large, and  $(\varepsilon_i)$ and $\varepsilon$ can be taken arbitrarily small and  $p>2$, it follows that the right-hand-side of the above inequality goes to infinity when $k\to \infty$, which yields a contradiction.
\end{proof}

Recall that a Banach space $X$ is said to be primary if whenever $X$ is isomorphic to $Y\oplus Z$, then either $Y$ or $Z$ is isomorphic to $X$.
It is known that
$L_{p,q}(0,1)$ and $L_{p,q}(0,\infty )$ are primary spaces (see \cite[Theorem  2.d.11]{LT2} and \cite[Theorem A.1]{Dilworth90}).
This property of $L_{p,q}$ function spaces is very useful in the study of the isomorphic embedding of general $L_{p,q}$-spaces.
\begin{rem}
There exist $\sigma$-finite measure spaces $(\Omega, \Sigma,\mu)$ such that $L_{p,q}(\Omega)$ is not primary.
Indeed,
consider $L_{p,q}(0,1)\oplus \ell_{p,q}$.
By Theorem \ref{ifonlyif2}, we have
  $\ell_{p,q}\not\approx L_{p,q}(0,1)\oplus \ell_{p,q}$.
By the fact that $\ell_{p,q}\not\hookrightarrow L_{p,q}(0,1)$\cite{KS,SS2018}, we obtain that
  $L_{p,q}(0,1)\not\approx L_{p,q}(0,1)\oplus \ell_{p,q}$.
  We conclude that
$L_{p,q}(0,1)\oplus \ell_{p,q}$ is not a primary space.

\end{rem}

We
show below  that the atomic part of $L_{p,q}$ can be `absorbed' by the atomless part.
\begin{proposition}\label{5.5} Let $1<p<\infty$, $1\le q<\infty$, $p\ne q$.
For any finite measure space $(\Omega, \Sigma, \mu)$ which is not purely atomic,
we have
 $L_{p,q}(\Omega)\approx L_{p,q}(0,1)$.
 \end{proposition}
 \begin{proof}
Letting  $\Omega =\Omega_1\oplus \Omega_2$ be the decomposition such that $\Omega_1$ is atomless and $\Omega_2$ is  atomic,
we have
$L_{p,q}(\Omega_2)\xhookrightarrow{c} L_{p,q}(0,1)$.
Hence,
there exists a Banach space $X$ isomorphic to $L_{p,q}(0,1)$ such that $X\oplus L_{p,q}(\Omega_2) \approx L_{p,q}(0,1)$.
Since $L_{p,q}(0,1)$ is a primary Banach space and $L_{p,q}(\Omega_2) \not \approx L_{p,q}(0,1)$ (see Theorem \ref{ifonlyif2}), it follows that $X\approx L_{p,q}(0,1)$.
Hence,
$$L_{p,q}(\Omega_1)\approx L_{p,q}(0,1)\approx X .$$
We obtain that $L_{p,q}(\Omega )\approx L_{p,q}(\Omega_1) \oplus L_{p,q}(\Omega_2 ) \approx X \oplus L_{p,q}(\Omega_2 ) \approx L_{p,q}(0,1)$.
\end{proof}

It is known that $\ell_{p,q}$ is not isomorphic to  a subspace of $L_{p,q}(0,1)$\cite{KS,SS2018}.
We strengthen  this result by showing   that $\ell_{p,q}$ can not be embedded into $\ell_{p,q,\infty}\oplus L_{p,q}(0,1)$.
\begin{proposition}\label{lpqtoinfty}Let $1<p<\infty$, $1\le p<\infty$ and $p\ne q$.
Then,
$$\ell_{p,q} \not\hookrightarrow \ell_{p,q,\infty} \oplus L_{p,q}(0,1) .$$
In particular, $U_{p,q} \not\hookrightarrow \ell_{p,q,\infty} \oplus L_{p,q}(0,1)$.
\end{proposition}
\begin{proof}
Assume that there exists an isomorphism  $T: \ell_{p,q}  \to \ell_{p,q,\infty} \oplus L_{p,q}(0,1)$.
Let $P_1$ and $P_2$ be the projection of $\ell_{p,q,\infty} \oplus L_{p,q}(0,1)$ onto $\ell_{p,q,\infty} $
and $  L_{p,q}(0,1)$, respectively.

Let $\{e^{\ell_{p,q} }_n\}$ be the natural basis of $\ell_{p,q}$.
If $P_1 T (e_{n}^{\ell_{p,q}})\to 0$ in $\norm{\cdot}_{p,q}$.
Then, passing to a subsequence, we may assume that $ \{P_2T (e_{n}^{\ell_{p,q}})\} \sim \{e_n^{\ell_{p,q}}\}$. This implies that $\ell_{p,q}\hookrightarrow L_{p,q}(0,1)$, which is a contradiction \cite{KS,SS2018}.

Passing to a subsequence if necessary, we may assume that  $\norm{P_1 T (e_{n}^{\ell_{p,q}})}_{p,q} \ge \delta >0$ for all $n$.
Passing to a subsequence if necessary, we may assume that $\{P_1 T (e_{n}^{\ell_{p,q}})\}$ is equivalent to a block basic sequence  in $\ell_{p,q,\infty}$.
By the definition of $\ell_{p,q,\infty}$, we obtain that   $P_1 T (e_{n}^{\ell_{p,q}})\to 0$ in the uniform norm.
By Lemma \ref{2.1}, passing to a subsequence if necessary, we obtain that $ \{P_1T (e_{n}^{\ell_{p,q}})\} \sim \{e_n^{\ell_{ q}}\}$.
Hence, $$\overline{{\rm span}}^{\norm{\cdot}_{\ell_{p,q}}} \{e_n^{\ell_{p,q}} \}\hookrightarrow \overline{{\rm span}}^{\norm{\cdot}_{\ell_{p,q,\infty }}} \{P_1 T(e_n^{\ell_{p,q}} )\} \oplus \overline{{\rm span}}^{\norm{\cdot}_{L_{p,q}}} \{P_2 T( e_n^{\ell_{p,q}}  )\} \hookrightarrow   \ell_q\oplus L_{p,q}(0,1) . $$
By Proposition \ref{5.5} above, we obtain that $L_{p,q}(0,1)\oplus \ell_q \approx L_{p,q}(0,1)$.
Therefore, $\ell_{p,q}\hookrightarrow L_{p,q}(0,1)$,  which is a contradiction \cite{KS,SS2018}.
\end{proof}

We prove the
  following far-reaching generalization of  \cite[Theorem 11]{KS} and \cite[Theorem 10]{SS2018} by using  different techniques from that in \cite{KS,SS2018}.
  In particular,
   the techniques used in \cite{KS,SS2018}  rely on the values of $p$ and $q$,
   our approach used in Theorem \ref{5.10} below is much simpler and does not rely on the values of $p$ and $q$.
\begin{theorem}\label{5.10}Let $1<p<\infty$, $1\le p<\infty$ and $p\ne q$.
For any $\sigma$-finite atomic measure space $(\Omega,\Sigma, \mu)$ of type  $\ell_{p,q,1}$ or $\ell_{p,q, \infty}$,
$$U_{p,q} \not\hookrightarrow L_{p,q}(0,1)\oplus \ell_{p,q}(\Omega).$$
In particular, $L_{p,q}(0,\infty ) \not\hookrightarrow L_{p,q}(0,1)\oplus \ell_{p,q}(\Omega).$
\end{theorem}
\begin{proof}
By Proposition \ref{lpqtoinfty}, it suffices to prove that
$U_{p,q} \not\hto L_{p,q}(0,1)\oplus \ell_{p,q}$.
Without loss of generality, we assume that   $U_{p,q} =\ell_{p,q}(\{A_n\})$, where
$\{A_n\}=\cup_n B^n:=\cup \{B^n_j\}_{j\ge 1}$ and $B^ n$ consists of  infinitely many atoms with measure $\frac{1}{n}$ for each $n$.
Define $e_{j}^n$ be the normalized basis of the characterized function $\chi_{B_j^n}$.

Assume that
there exists an isomorphism embedding
$$T:U_{p,q}  \to  L_{p,q}(0,1)\oplus \ell_{p,q} .$$
Let $P_1$ and $P_2$ be  projections  from $L_{p,q}(0,1)\oplus \ell_{p,q} $ onto $L_{p,q}(0,1)$ and $ \ell_{p,q}$, respectively.

We claim that for any $1\le n_1 <n_2 < \cdots <n_n$, $n\ge 1$,
\begin{align}\label{P2TENJ}
\norm{P_2T \left(  \norm{\chi_{(0,\frac1n)}}_{p,q}  \sum_{1\le j\le n}e_{n_j}^n  \right)}_{p,q}\not \to 0
\end{align}
as $n\to \infty$.
Indeed, otherwise, passing to a subsequence of $\{n\}$ if necessary, we have
$$ \{ e_n^{\ell_{p,q}} \}
\sim
\left\{ \norm{\chi_{(0,\frac1n)}}_{p,q}  \sum_{1\le j \le n } e_{n_j}^n    \right\}
\sim
\left \{  P_1T  \left(
 \norm{\chi_{(0,\frac1n)}  \sum_{1\le j \le n } e_{n_j}^n  } \right)
  \right\} \hookrightarrow  L_{p,q}(0,1), $$
which is a contradiction to the fact that $\ell_{p,q}\not \hookrightarrow L_{p,q}(0,1)$\cite{KS,SS2018}.

If, for each $n $, $$ \liminf_{j\to \infty} \norm{ P_2T(e^n _j)}_{p,q }\to 0, $$
then,   passing to a subsequence if necessary, we obtain that\cite[Theorem 1.3.9]{AK}
$$   \{e_{j}^n\}_{n,j} \sim  \{P_1T(e_{j}^n)\}_{n,j}  .  $$
The structure of the elements $\{e_{n_j}^i  \}$ is shown in the following matrix:
\begin{equation*}
\left(
\begin{matrix}
e_{n_1}^1 & \cdots &  e_{n_2}^1   &  \cdots  &  e_{n_4}^1  & \cdots&\cdots  & e_{n_7}^1 &\cdots&\cdots&  \cdots&   e_{n_{11}}^1&  \cdots&  \cdots\\
0  & \cdots & \cdots  &  e_{n_3}^2   & \cdots &  e_{n_5}^2&\cdots &\cdots  & e_{n_8}^2&\cdots&  \cdots &  \cdots&  e_{n_{12}}^2&  \cdots\\
0 & \cdots & \cdots    & \cdots & \cdots &  \cdots &  e_{n_6}^3 & \cdots & \cdots   & e_{n_9}^3  &  \cdots&  \cdots&  \cdots
\\
0 & \cdots & \cdots    & \cdots & \cdots &  \cdots &  \cdots & \cdots & \cdots   &  \cdots &  e_{n_{10}}^4&  \cdots &  \cdots
\\
0 & \cdots & \cdots    & \cdots & \cdots &  \cdots &  \cdots & \cdots & \cdots   &  \cdots &   \cdots &  \cdots&  \cdots
\end{matrix}\right).
\end{equation*}
This implies that $U_{p,q} \hookrightarrow L_{p,q}(0,1)$, which is impossible \cite{SS2018,KS}.

Passing to a subsequence if necessary, we assume that
for each $n$,   there exists $\delta_n>0$ such that
$$ \liminf_{j\to \infty} \norm{ P_2T(e^n _j)}_{p,q }>  \delta_n. $$
Let \begin{align}\label{fnjp2}
f^n_j := \frac{P_2T (e^n_j)}{\norm{P_2T (e^n_j) }_{p,q}}.
 \end{align}In particular, $\norm{f_j^n}_{p,q}=1$ and $\norm{f_{j}^n}_\infty \le 1 $.
For $f^1_1$ and any $\varepsilon>0$, there exists a finitely supported element $g_1^1:= f_{n_1}^1 s(g_1^1) $ (where $s(g_1^1)$ is the support of $s(g_1^1)$) such that
$\norm{f_{n_1}^1 -g_1}_{p,q}\le \varepsilon_1.$
In particular, $\norm{g_{1}^1}_\infty \le 1 $.

Let $\varepsilon_2>0$.
Since $\{e^2_j\}_j$ is weakly null, it follows that there exist  $f^2_{n_1}$ and  a finitely support element $g_1^2:= f_{n_1}^2 s(g_1^2) $ (where $s(g_1^2)$ is the support of $s(g_1^1)$) such that
 $\norm{f^2_{n_1} -g^2_1}_{p,q}\le \varepsilon_2 $
and $g_1^1 g_1^2=0$.

Let $\varepsilon_3>0$.
Since $\{e^2_j\}_j$ is weakly null, it follows that there exist  $f^2_{n_2}$ and   a finitely support element $g_2^2:= f_{n_2}^2 s(g_2^2) $ (where $s(g_2^2)$ is the support of $s(g_2^2)$) such that
 $\norm{f^2_{n_2} -g^2_2}_{p,q}\le \varepsilon_3 $
and $g_1^1 g_2^2=0=g_1^2g_2^2 $.

Constructive inductively, we obtain a sequence $\{g^n_j\}_{1\le j\le n,n\ge 1 }$ equivalent to $\{f^n_{n_j}\}_{1\le j\le n,n\ge 1}$\cite[Theorem 1.3.9]{AK}.
The structure of the elements $\{g_{ j}^n  \} _{1\le j\le n,n\ge 1 } $ is shown in the following matrix:
\begin{equation*}
\left(
\begin{matrix}
g_{ 1}^1 &    \cdots   &  \cdots  &  \cdots & \cdots&\cdots  & \cdots &\cdots&\cdots&  \cdots&   \cdots&  \cdots&  \cdots&\cdots   \\
\cdots  & \cdots   &  g_{ 1}^2   & \cdots &  g_{2}^2&\cdots &\cdots  & \cdots &\cdots&  \cdots &  \cdots&  \cdots &  \cdots& \cdots   \\
\cdots & \cdots      & \cdots & \cdots &  \cdots &  \cdots &  g_{1}^3 & \cdots & g_{ 2}^3 & \cdots  & g_{ 3}^3 & \cdots  & \cdots  &\cdots
\\
\cdots & \cdots     & \cdots & \cdots &  \cdots &  \cdots & \cdots & \cdots &  \cdots  &  \cdots &\cdots  &\cdots  &  g_{1}^4   &\cdots
\\
\cdots  & \cdots     & \cdots & \cdots &  \cdots &  \cdots & \cdots & \cdots &  \cdots  &  \cdots &   \cdots & \cdots  & \cdots  &\cdots
\end{matrix}\right).
\end{equation*}
Since $g_{j}^n$ are disjointly supported and $\norm{g_{j}^n}_\infty \le 1$, it follows
that
\begin{align*}
\norm{
\frac{
    \sum_{1\le j \le n}  \norm{P_2T  (e_{n_j}^n) }_{p,q} g_j^n  }
    {
    \norm{
   \sum_{1\le j \le n}  \norm{P_2T  (e_{n_j}^n) }_{p,q} g_j^n }_{p,q}
    }
 }_\infty
&\sim
\norm{
\frac{
    \sum_{1\le j \le n}  \norm{P_2T  (e_{n_j}^n) }_{p,q} g_j^n  }
    {
    \norm{
   \sum_{1\le j \le n}  \norm{P_2T  (e_{n_j}^n) }_{p,q} f_{n_j}^n }_{p,q}
    }
 }_\infty\\
 &\stackrel{\eqref{fnjp2}}{=}
 \norm{
\frac{
    \sum_{1\le j \le n}  \norm{P_2T  (e_{n_j}^n) }_{p,q} g_j^n  }
    {
    \norm{
   \sum_{1\le j \le n}   P_2T  (e_{n_j}^n)  }_{p,q}
    }
 }_\infty
 \stackrel{\eqref{P2TENJ}}{\lesssim} \frac{ \norm{P_2T}}{ \frac{\delta}{ \norm{\chi_{(0,\frac1n)}}_{p,q}}} \to 0
\end{align*}
as $n \to \infty$.

By Lemma \ref{2.1}, there exists a subsequence $\{n_k\}$ of $\{n\}$  such that   $\left\{\frac{
    \sum_{1\le j \le {n_k}}  \norm{P_2T  (e_{n_j}^{n_k}) }_{p,q} g_j^{n_k}  }
    {
    \norm{
   \sum_{1\le j \le {n_k}}  \norm{P_2T  (e_{n_j}^{n_k}) }_{p,q} g_j^{n_k} }_{p,q}
    } \right\} \sim\{e_{k}^{\ell_q}\}
$.
Note that
\begin{align*}
&\qquad \overline{{\rm span}}^{\norm{\cdot}_{\ell_{p,q}}}  \left\{e_{n}^{p,q}\right\}_n
 \approx   \overline{{\rm span}}^{\norm{\cdot}_{\ell_{p,q}}}  \left\{ \sum_{1\le j \le {n_k }} e_{j}^{n_k }\right\}_k
\approx   \overline{{\rm span}}^{\norm{\cdot}_{L_{p,q} \oplus \ell_{p,q}}} \left \{T(\sum_{1\le j \le {n_k }} e_{j}^{n_k })\right\}_k  \\
&\hookrightarrow   \left(\overline{{\rm span}}^{\norm{\cdot}_{L_{p,q}}}  \left\{ P_1 T (\sum_{1\le j \le {n_k }  } e_{j}^{n_k }  )\right\}_k\right)  \bigoplus    \left(\overline{{\rm span}}^{\norm{\cdot}_{\ell_{p,q}}}  \left\{P_2 T(\sum_{1\le j \le {n_k } } e_{j}^{n_k }  )\right\}_k\right)\\
&\hookrightarrow
L_{p,q}(0,1)\oplus \overline{{\rm span}}^{\norm{\cdot}_{\ell_{p,q}}}
  \left\{
 \sum_{1\le j \le {n_k }}  \norm{P_2T  (e_{n_j}^{n_k}) }_{p,q}  f_{n_j}^{n_k }
 \right\}_k  \\
&
\approx L_{p,q}(0,1)\oplus \overline{{\rm span}}^{\norm{\cdot}_{\ell_{p,q}}}  \left\{\frac{\sum_{1\le j \le n_k }   \norm{P_2T  (e_{n_j}^{n_k}) }_{p,q}   g_j^{n_k }}{\norm{\sum_{1\le j \le {n_k }}  \norm{P_2T  (e_{n_j}^{n_k}) }_{p,q}   g_j^{n_k }}}\right\}_k  \\
& \hookrightarrow  L_{p,q}(0,1)\oplus \ell_q \stackrel{\tiny \mbox{Prop. \ref{5.5}}}{\approx} L_{p,q}(0,1).
\end{align*}
This is a contradiction with the fact that $\ell_{p,q}\not\hookrightarrow L_{p,q}(0,1)$\cite{KS,SS2018}.
\end{proof}

The following corollary yields the results  of  \cite[Theorem 11]{KS} and \cite[Theorem 10]{SS2018} in a more general setting and  in much simpler fashion.
\begin{corollary}\label{infinityinto}Let $1<p<\infty$, $1\le p<\infty$ and $p\ne q$.
For any $\sigma$-finite atomic measure space $(\Omega,\Sigma, \mu)$, we have
$$L_{p,q}(0,\infty)\not\hookrightarrow L_{p,q}(0,1)\oplus \ell_{p,q}(\Omega).$$
In particular,
for any atomless finite  measure space  $(\Omega_1,\Sigma_1, \mu_1)$ and atomic  infinite $\sigma$-finite  measure space  $(\Omega_2,\Sigma_2, \mu_2)$,
 we have
 $$L_{p,q}(\Omega_1\oplus \Omega_2)\not\approx L_{p,q}(0,\infty)$$
\end{corollary}
\begin{proof} By Proposition \ref{0c}, it suffices to prove that $$L_{p,q}(0,\infty)\not\hookrightarrow L_{p,q}(0,1)\oplus U_{p,q}.$$
Assume by contradiction that there exists an isomorphic embedding
 $$T: L_{p,q}(0,\infty)  \hookrightarrow L_{p,q}(0,1)\oplus U_{p,q}.$$

For every $n,k\in \mathbb{N}$, $t\in (0,\infty)$, we define
$$
r_{n,k}(t):=
\begin{cases}r_n(t-k+1), &t\in (k-1,k];\\
0, &\mbox{elsewhere},
\end{cases}
$$
 where $r_n$, $n\in \mathbb{N}$, are the Rademacher functions.
 By \cite[Theorem 5]{KS}, $\{r_{n,k}\}_{k,n=1}^\infty $ is a basic sequence, and hence,
$\{T(r_{n,k})\}_{k,n=1}^\infty $
is also a basic sequence with a  basic constant, which does not exceed $\left\|T\right\|$.
Let $P_1 $ and $P_2$ be projections from $L_{p,q}(0,1)\oplus U_{p,q}$
onto $L_{p,q}(0,1)\oplus 0$ and $0\oplus U_{p,q}$, respectively.

Consider the sequence $\{P_2T(r_{n,k})\}_{n=1}^\infty$, $k\in \mathbb{N}$.
Without loss of generality, we may assume that
$$
\inf_n \norm{P_2T(r_{n,k})}_{p,q} =0,~\forall k\in \mathbb{N},$$
or else
$$
\inf_n \norm{P_2T(r_{n,k})}_{p,q}
>0,~\forall k\in \mathbb{N}.$$

If $
\inf_n \norm{P_2T(r_{n,k})}_{p,q} =0,~\forall k\in \mathbb{N},$
it is immediate that there exists a sequence $\{n_k\}_{k=1}^\infty$ of integers such that
$$\{T(r_{n_k, k})\}_{k=1}^\infty \sim
\{P_1T(r_{n_k, k})\}_{k=1}^\infty.$$
Hence,
$$\{e^{\ell_{p,q}}_k\}_{k=1}^\infty
\sim
\{r_{n_k, k}\}_{k=1}^\infty \sim
\{T(r_{n_k, k})\}_{k=1}^\infty \sim
\{P_1T(r_{n_k, k})\}_{k=1}^\infty. $$
¡¡In particular, $\ell_{p,q}\hookrightarrow L_{p,q}(0,1)$,
which contradicts with that fact that $\ell_{p,q}\not\hookrightarrow L_{p,q}(0,1)$ (see \cite{KS,SS2018}).

Now suppose that $
\inf_n \norm{P_2T(r_{n,k})}_{p,q}
>0,~\forall k\in \mathbb{N}.$
Without loss of generality, we may assume that
$$\inf_k \left(
\inf_n \norm{P_2T(r_{n,k})}_{p,q}\right)=0,$$
or else
$$\inf_k \left(
\inf_n \norm{P_2T(r_{n,k})}_{p,q}\right)>0.$$

If $\inf_k \left(
\inf_n \norm{P_2T(r_{n,k})}_{p,q}\right)=0,$
then it is immediate  that
 then there exist two
sequences $\{n_k\}_{k=1}^\infty$ and $\{j_k\}_{k=1}^\infty $
 of integers such that
$$\{e^{\ell_{p,q}}_k\}_{k=1}^\infty
\sim
\{r_{n_k,j_k}\}_{k=1}^\infty \sim
\{T(r_{n_k,j_k})\}_{k=1}^\infty \sim
\{P_1T(r_{n_k,j_k})\}_{k=1}^\infty. $$
¡¡In particular, $\ell_{p,q}\hookrightarrow L_{p,q}(0,1)$,
which contradicts with that fact that $\ell_{p,q}\not\hookrightarrow L_{p,q}(0,1)$ (see \cite{KS,SS2018}).

Now, suppose that
$\inf_k \left(
\inf_n \norm{P_2T(r_{n,k})}_{p,q}\right)>0 $.
Since the sequence
$\{P_2T(r_{n,k})\}_{n=1}^\infty $ is weakly null\cite[Proposition 2.c.10]{LT2}, without loss of generality, by \cite[Proposition 1.a.12]{LT1}, we may
assume that
for any fixed $k\in \mathbb{N}$, the sequence $\{P_2T(r_{n,k})\}_{n=1}^\infty $
is equivalent to a sequence of disjointly supported
elements in $U_{p,q}$.
By \cite[Lemma 2]{KS}, there exist two sequences
$\{a_{k}\}_{k=1}^\infty $
and $\{b_{k}\}_{k=1}^\infty $,
 and a subsequence of $\{r_{n,k}\}_{n=1}^\infty$ (for simplicity, still denote by $\{r_{n,k}\}_{n=1}^\infty$)
such that for
$$F_k:= \frac{\sum_{n=a_k}^{b_k}P_2T(r_{n,k}) }{\norm{\sum_{n=a_k}^{b_k}P_2T(r_{n,k})}_{p,q}} \in U_{p,q}, k\ge 1, $$
the convergence $ F_k^*\to 0$ holds as $k\to \infty $.
By \cite[Lemma 1]{KS}, passing to a subsequence if necessary,
we may assume that
$$\{F_{k}\}_{k=1}^\infty \sim \{e^{\ell_q}_k\}_{k=1}^\infty . $$
On the other hand, by
\cite[Theorem 6]{KS},
for the elements,
$$G_k:= \frac{\sum_{n=a_k}^{b_k}r_{n,k}}{\norm{\sum_{n=a_k}^{b_k}r_{n,k} }_{p,q}}\in [r_{n,k}]_{n=1}^\infty,$$
we have
$$\{G_{k}\}_{k=1}^\infty \sim\{e^{\ell_{p,q} }_k \}^\infty_{k=1}.$$
Denote $\alpha_k:=
\frac{
 \norm{
 \sum_{n=a_k}^{b_k}P_2T(r_{n,k})}_{p,q}
 }{ \norm{\sum_{n=a_k}^{b_k} r_{n,k} }_{p,q
}}$. In particular, $\sup_k \alpha_k \le \norm{P_2T}$.
We have
$$P_2T(G_k) =\alpha F_k, ~k\ge 1. $$
Passing to a subsequence if necessary, we  may assume that $\alpha_k\to 0$ as $k\to \infty $ or else $\inf_k\alpha_k >0$.
If $\alpha_k\to 0$ as $k\to \infty$, then $\norm{P_2T(G_k)}_{p,q}\to 0$ and hence, passing to a subsequence if necessary,
$ \{P_1T(G_k)\}_{j=1}^\infty
\sim
\{T(G_k) \}_{j=1}^\infty$.
This implies that $\ell_{p,q}\hookrightarrow L_{p,q}(0,1)$, which is impossible (see \cite{KS,SS2018}).

Now, if $\inf_{k}\alpha_k >0$, then we have (note that both $\{P_1T(G_k)\}_{k=1}^\infty $
 and $\{P_2T(G_k)\}_{k=1}^\infty$ are basic sequences\cite[Proposition 1.5.4]{AK})
\begin{align*}
[e_k^{\ell_{p,q}}]_{k=1}^\infty
\sim
[G_k]_{k=1}^\infty
& \sim
[P_1T(G_k) +P_2T(G_k)]_{k=1}^\infty
\hookrightarrow  [P_1T(G_k) ]_{k=1}^\infty\oplus  [P_2T(G_k)]_{k=1}^\infty \\ &\approx [P_1T(G_k) ]_{k=1}^\infty\oplus  [F_k]_{k=1}^\infty \approx L_{p,q}(0,1)\oplus \ell_q \stackrel{\tiny Prop. \ref{5.5}}{\approx} L_{p,q}(0,1),
\end{align*}
which is impossible
(see \cite{KS,SS2018}).
\end{proof}

\begin{cor}\label{primary}Let $1<p<\infty$, $1\le p<\infty$ and $p\ne q$.
 Let $(\Omega,\Sigma, \mu)$ be an infinite $\sigma$-finite measure space and assume that   $\Omega$ can be  decomposed  as $\Omega_1\oplus \Omega_2$, where  $\Omega_1$ is atomless with  $\mu(\Omega_1)=\infty$  and $\Omega_2$ is atomic.
 We have
 $$L_{p,q}(\Omega)\approx L_{p,q}(0,\infty).$$
 \end{cor}
 \begin{proof}
Since $L_{p,q}(\Omega_2)\xhookrightarrow{c} L_{p,q}(0,\infty)$, it follows that there exists a Banach space $X\approx L_{p,q}(0,\infty) $ such that
$$X\oplus L_{p,q}(\Omega_2) \approx L_{p,q}(0,\infty) . $$
Moreover, since $L_{p,q}(0,\infty)$ is primary and $L_{p,q}(\Omega_2 )\not\approx L_{p,q}(0,\infty)$ (see Corollary \ref{infinityinto}), it follows that
 $X \approx L_{p,q}(0,\infty)$.
 Hence,
 $$L_{p,q}(\Omega )\approx L_{p,q}(\Omega_1) \oplus L_{p,q}(\Omega_2) \approx L_{p,q}(0,\infty)\oplus L_{p,q}(\Omega_2)  \approx X \oplus L_{p,q}(\Omega_2) \approx L_{p,q}(0,\infty)  ,$$
 which completes the proof.
 \end{proof}

\appendix
\section{Convergence of   Rademacher series}
In this appendix, we establish a fact used in the proof of Theorem \ref{ifonlyif2}.
Let $G:= ({\rm exp} (L_2))_0$ be the separable part of the exponential Orlicz space defined by Orlicz function $\Phi_2(t):= e^{t^2} -1$ (see e.g.  \cite{JSZ,RS,AS05}).
The classical Khintchine inequality was extended by   Rodin and Semenov\cite{RS}.
Precisely,  for any symmetric function space\cite{KPS} (called  r.i space in \cite{LT2}) $E(0,1)$ containing $G$,
we have
\begin{align}\label{RS}
c \norm{\alpha}_2 \le \norm{\sum_{k\ge 0}  \alpha_k r_k }_E \le C  \norm{\alpha}_2, ~\alpha:=(\alpha_1,\alpha_2,\cdots)\in \ell_{2},
\end{align}
where $\{r_k\}$ is the Rademacher system.

We say that a sequence of bounded elements $\{x_n\}$ from a symmetric function space $E(0,1)$ has absolutely equi-continuous norms\cite{CS94,SC90} if
$$\overline{\lim\limits_{n}} (\sup_{m\ge 1}\norm{x_m \chi_{A_n}}_{E})=0 $$
for any decreasing sequence of $A_n\subset [0,1]$ such that $\mu(A_n)\to 0$.
\begin{lem}\label{CS}
\cite[Proposition 3.1]{CS94}
Let $E(0,1)$ be a separable symmetric function space and let $\{x_n\}$ be a sequence of bounded elements  in $E(0,1)$ having absolutely equi-continuous norms and $x_n\to x\in E(0,1)$ in measure,
then $$\norm{x_n-x}_E \to 0. $$
\end{lem}

The theorem below is folklore.
However, due to the lack of references, we provide a short proof.

\begin{theorem}\label{convertoG}
Let $E(0,1)$ be a separable symmetric function space on $(0,1)$ such that
$E(0,1)\supset G$,
then
$$\norm{\left(\frac{1}{\sqrt{n}}\sum_{k=1}^n r_n\right)^*  -\xi ^*  }_E \to 0,$$
where $ \xi$ is a normal distribution which satisfies that
$\mu( \xi \in A ) = \frac{1}{\sqrt{2\pi }} \int_A e^{-t^2/2} dt$ for any interval $A\subset [0,\infty )$.
\end{theorem}
\begin{proof}
The central limit theorem shows that
$f_n:=\frac{1}{\sqrt{n}}\sum_{k=1}^n r_n\to \xi$ as $n\to \infty$ in distribution\cite[Theorem 27.1]{Billingsley}, and therefore,
$$f_n^* \to \xi^* $$
 in measure.
By \eqref{RS}, we obtain that $\sup_n \norm{f_n}_E <\infty $ and $ \sup_n \norm{f_n}_G <\infty $. We claim that
$$f_n^*(t) \le {\rm Const}  \cdot (\log(e/t))^{1/2}. $$
Indeed, by   Lemma 4.3  in \cite{AS05}, we obtain
$$1 \stackrel{\eqref{RS}}{\approx} \norm{ f_n  }_G \approx \sup _{t\in (0,1)}  \frac{t}{t [\log(e/t)] ^{1/2}} f_n^*(t), $$
that is, $f_n^*(t) \lesssim  [\log(e/t)] ^{1/2 }\in G $.
This implies that $\{f_n^*\}_{n\ge 1}$ has absolutely equi-continuous norms.
Therefore, by Lemma \ref{CS}, we obtain that  $$\norm{f^*_n -\xi}_E \to_n  0, $$
which completes the proof.
\end{proof}



\begin{thebibliography}{99}

\bibitem{AK}
F. Albiac, N. Kalton,
{\it Topics in Banach space theory,}
Graduate Texts in Mathematics 233, Springer, 2006.

\bibitem{ACL}
Z. Altshuler, P. Casazza, B.-L.  Lin,
{\it On symmetric Basic sequences in Lorentz sequence spaces},
Israel. J. Math. \textbf{15} (1973), 144--155.

\bibitem{Arazy81}
J. Arazy,
 {\it Basic sequences, embeddings, and the uniqueness of the symmetric structure in unitary matrix spaces,}
  J. Funct. Anal. {\bf 40} (1981), no. 3, 302--340.


\bibitem{AHS}
S. Astashkin, J. Huang, F. Sukochev,
{\it  Lack of isomorphic embeddings of symmetric function spaces into operator ideals, }
J. Funct. Anal. 280 (2021), no. 5, 108895, 34 pp.

\bibitem{ASS}
S. Astashkin, E. Semenov, F. Sukochev,
{\it The Banach--Saks $p$-property},
Math. Ann. \textbf{332} (2005), 879--900.

    \bibitem{AS05}
    S. Astashkin, F. Sukochev,
    {\it Series of independent random variables in rearrangement invariant space,}
    Israel J. Math. \textbf{145} (2005), 125--156.

 \bibitem{Banach}
   S. Banach, {\it Th\'{e}orie des op\'{e}rations lin\'{e}aires,}
   Warszawa (1932).
    \bibitem{Bennett_S} C. Bennett, R. Sharpley, {\it Interpolation of operators,} Academic Press, Boston, 1988.

\bibitem{Billingsley}
P. Billingsley,
{\it Probability and measure,}
3rd edition,
The University of Chicago,
John Wiley $\&$ Sons, New York, 1995.

\bibitem{Bylinkina}
O.P. Bylinkina,
{\it Some estimates of the Banach--Mazur   distance between finite-dimensional $L_{p,q}$ spaces},
Theory of operators in function spaces, 14--28,
Voronezh. Gos. Univ. Voronezh, 1983. (Russian)

\bibitem{Bogachev}
V.I. Bogachev, {\it Measure theory,}
Vol. I, II.
Springer-Verlag, Berlin, 2007.





\bibitem{CD85}
N. Carothers, S. Dilworth,
{\it Geometry of Lorentz spaces via interpolation},
Texas Functional Analysis Seminar 1985--1986 (Austin, TX, 1985--1986), 107--133, Longhorn Notes, Univ. Texas, Austin, TX, 1986.

\bibitem{CD}
N. Carothers, S. Dilworth,
{\it Subspaces of $L_{p,q}$,}
Proc. Amer. Math. Soc.
{\bf 104} (1988), no. 2, 537--545.

\bibitem{CD89}
N. Carothers, S. Dilworth,
{\it Equidistributed random variables in $L_{p,q}$,}
    J. Funct. Anal.
    \textbf{84} (1989), 146--159.

\bibitem{CF}
N. Carothers, P. Flinn,
{\it Embedding $l_p^{n^\alpha}$ in $l_{p,q}^n$},
Proc. Amer. Math. Soc. \textbf{88} (1983), 523--526.


\bibitem{CS94}
V. Chilin, F. Sukochev,
{\it Weak Convergence in non-commutative symmetric spaces,}
J. Operator Theory  \textbf{31} (1994), 35--65.



\bibitem{Dilworth90}
S. Dilworth,
{\it A scale of linear spaces related to the $L_p$ scale,}
Illinois J. Math.
\textbf{34} (1) (1990), 140--148.

  \bibitem{Dilworth}
    S. Dilworth,
 {\it Special Banach lattices and their applications,} Handbook of the geometry of Banach spaces, Vol. I, 497--532, North-Holland, Amsterdam, 2001.

\bibitem{DDP} P.G.~Dodds, T.K.-Y.~Dodds, B.~Pagter, {\it Non-Commutative Banach Function Spaces}, Math.\ Z. {\bf 201} (1989), 583--597.

\bibitem{DDS}
 P. Dodds, T. Dodds, F. Sukochev,
 {\it Banach--Saks properties in symmetric spaces of measurable operators},
 Studia Math.
  {\bf 178} (2007), 125--166.

  \bibitem{DDS2014}
   P. Dodds, T. Dodds, F. Sukochev,
   {\it On $p$-convexity and $q$-concavity in non-commutative symmetric spaces},
Integr. Equ. Oper. Theory \textbf{78} (2014), 91--114.

\bibitem{DFPS}
 P. Dodds, S.  Ferleger, B. de Pagter,  F. Sukochev,
  {\it Vilenkin systems and generalized triangular truncation operator,} Integr. Equ. Oper. Theory
   {\bf 40} (2001), no. 4, 403--435.



\bibitem{DSS}
P. Dodds, E. Semenov, F. Sukochev,
{\it The Banach--Saks properties in rearrangement invariant spaces,}
 Studia
Math. \textbf{162} (2004), 263--294.

\bibitem{FLM}
T. Figiel, J. Lindenstrauss, V.D. Milman,
{\it The dimension of almost
spherical sections of convex bodies},
 Acta Math. \textbf{139}  (1977), 53--94.




\bibitem{HSS}
J. Huang, O. Sadovskaya, F. Sukochev,
{\it On Arazy's problem concerning isomorphic embeddings of ideals of compact operators,}
submitted manuscript.





\bibitem{JSZ}
Y. Jiao, F. Sukochev,  D. Zanin,
{\it Sums of independent and freely independent identically distributed random variables,}
Studia Math. \textbf{255} (2020), 55--81.

\bibitem{JMST}
W. Johnson, B. Maurey, G. Schechtman, L. Tzafriri,
{\it Symmetric structures in Banach spaces, }
Mem. Amer. Math. Soc.  \textbf{19} (1979), no. 217, v+298 pp.


\bibitem{JS}
W. Johnson, G. Schechtman,
{\it Embedding $\ell_p^m $ into $\ell_1^n$},
Acta Math. \textbf{149} (1982), 71--85.

\bibitem{JS03}
W. Johnson, G. Schechtman,
{\it Very tight embeddings of subspaces of $L_p$, $1\le p<2$, into $\ell_{p}^n$},
Geom. Funct. Anal. \textbf{13} (2003), 845--851.

\bibitem{Kadec}
M.I. Kadec,
{\it Linear dimension of the spaces $L_p$ and $L_q$,}
Uspekhi Mat. Nauk, 13 (1958),  95--98
(in Russian).

\bibitem{KP}
M. Kadec, Pe{\l}czy\'{n}ski,
{\it Bases, lacunary sequences and complemented subspaces in the spaces $L_p$},
Studia Math. \textbf{21} (1961/1962), 161--176.


\bibitem{KamMal}
A.~Kaminska, L. Maligranda,
{\it On Lorentz spaces $\Gamma_{p,\omega}$,}
Israel J. Math. \textbf{140} (2004), 285--318.

\bibitem{K}
B. Kashin,
{\it Section of some finite-dimensional sets and classes of smooth
functions} (in Russian), Izv. Acad. Nauk. SSSR 41 (1977), 334--351.



\bibitem{KPS}
S. Krein, Y. Petunin, E. Semenov,
 {\it Interpolation of linear operators},
 Translations of Mathematical Monographs, Amer. Math. Soc.
  {\bf 54} (1982).

\bibitem{KS}
    A. Kuryakov, F. Sukochev,
    {\it Isomorphic classification of $L_{p,q}$-spaces},
    J. Funct. Anal. \textbf{269} (2015), 2611--2630.



\bibitem{Leung}
D.H. Leung,
{\it Isomorphism of certain weak $L^p$ spaces, }
Studia Math. \textbf{104} (1993), 151--160.

\bibitem{Leung2}
D.H. Leung,
{\it Isomorphic classification of atomic weak $L^p$ spaces,}
Interaction between functional analysis, harmonic analysis, and probability (Columbia, MO, 1994), 315--330,
Lecture Notes in Pure and Appl. Math., 175, Dekker, New York, 1996.

\bibitem{Leung3}
D.H. Leung,
{\it Purely non-atomic weak $L^p$ spaces,}
Studia Math. \textbf{122} (1997), 55--66.

\bibitem{LS}
D.H.  Leung, R. Sabarudin,
{\it The classification problem for nonatomic weak $L^p$ spaces},
J. Funct. Anal. \textbf{258} (2010), 373--396.


\bibitem{LT3} J.~Lindenstrauss, L.~Tzafriri, {\it On Orlicz sequence spaces. III}, Israel J. Math. \textbf{14} (1973), 368--389.

\bibitem{LT1}  J. Lindenstrauss, L.  Tzafriri,
 {\it Classical Banach spaces. I. Sequence spaces.}
 Ergebnisse der Mathematik und ihrer Grenzgebiete, Vol. 92. Springer-Verlag, Berlin-New York, 1977.

\bibitem{LT2} J.~Lindenstrauss, L.~Tzafriri, \emph{Classical Banach spaces. II. Function spaces,} Ergebnisse der Mathematik und ihrer Grenzgebiete,
     Vol. 97. Springer-Verlag, Berlin-New York, 1979.

\bibitem{LZ}
J.~Lindenstrauss, M. Zippin,
{\it Banach spaces with a unique unconditional basis,}
J. Funct. Anal. \textbf{3} (1969), 115--125.

\bibitem{LSZ}
S. Lord, F. Sukochev, D. Zanin,
 {\it Singular traces: Theory and applications,}
  De Gruyter Studies in Mathematics,  46. De Gruyter, Berlin, 2013.

\bibitem{Lorentz50}
G.G. Lorentz, {\it Some new functional spaces,}
Ann. of Math. \textbf{51} (1950), 37--55.

\bibitem{Lorentz51}
G.G. Lorentz,
{\it On the theory of spaces $\Lambda$},
Pcific J. Math. \textbf{1} (1951), 411--429.


\bibitem{Luxemburg}
W.A.J. Luxemburg,
{\it Rearrangement-invariant Banach function spaces,}
Proc. Symp. Analysis, Queen's univ.  (1967), 83--144.



\bibitem{Maligranda}
L. Maligranda,
{\it Type, cotype and convexity properties of quasi-Banach spaces,}
 in: M. Kato, L. Maligranda (Eds.), Banach
and Function Spaces, Proc. of the Internat. Symp. on Banach and Function Spaces, Kitakyushu--Japan, 2--4 Oct. 2003,
Yokohama Publ., 2004, pp. 83--120.




\bibitem{MP}
B. Maurey, G. Pisier,
{\it Series de variables al\'{e}atoires vectorielles independantes et propri\'{e}t\'{e}s geom\'{e}triques des espaces de Banach,}
Studia Math. \textbf{58} (1976), 45--90.

\bibitem{Novikova}
A. Novikova,
{\it The Banach--Saks index for Rademacher subspaces,}
Vestnik SSU, vol. 36, Samara University, Samara, 2005 (in Russian).


\bibitem{RS}
V. Rodin, E. Semenov,
{\it Rademacher series in symmetric spaces,}
Anal. Math. \textbf{1} (1975), 207--222.

\bibitem{Rosenthal}
H. Rosenthal,
{\it On the subspaces of
$L^p$ $(p>2)$ spanned by sequences of independent random variables,}
Israel J. Math. \textbf{8} (1970), 273--303.


    \bibitem{SS2018}
    O. Sadovskaya, F. Sukochev,
    {\it Isomorphic classification of $L_{p,q}$-spaces: the case $p=2$, $1\le q<2$,}
    Proc. Amer. Math. Soc. \textbf{146} (2018), 3975--3984.

\bibitem{Schutt}
C. Schutt,
{\it Lorentz spaces that are isomorphic to subspaces of $L^1$},
Trans. Amer. Math. Soc. \textbf{314} (1989), 583--595.

\bibitem{SS2004}
E. Semenov, F. Sukochev,
{\it The Banach--Saks index},
Mat. Sb. \textbf{195} (2004), 117--140.


\bibitem{S96}
F. Sukochev,
\textit{Non-isomorphism of $L_p$-spaces associated with finite and infinite von Neuman algebras,}
 Proc. Amer. Math. Soc. \textbf{124}(5)  (1996), 1517--1527.

    \bibitem{S01}
    F. Sukochev,
    {\it On the Banach-isomorphic classification of $L_p$ spaces of hyperfinite semifinite von Neumann algebra,}
     Geometric analysis and applications (Canberra, 2000), 213--221, Proc. Centre Math. Appl. Austral. Nat. Univ., 39, Austral. Nat. Univ., Canberra, 2001.

\bibitem{SC90}
F. Sukochev, V. Chilin,
{\it Symmetric spaces on semifinite von Neumann algebras,}
Soviet Math. Dokl. \textbf{42} (1991), 97--101.

     \bibitem{Tradacete}
     P. Tradacete,
     {\it Subspace structure of Lorentz $L_{p,q}$ spaces and strictly singular operators,}
     J. Math. Anal. Appl. \textbf{367} (2010), 98--107.

 \bibitem{Wojtaszczyk}
 P. Wojtaszczyk,
 {\it Banach spaces for analysts,}
  Cambridge Studies in Advanced Mathematics, 25. Cambridge University Press, Cambridge, 1991.

\end{thebibliography}
\end{document}